\documentclass[final, 11 pt]{article}

\usepackage{tikz}
\usepackage{upgreek,ctable}
\usepackage{bm} 
\usepackage{amsthm}
\usepackage{xcolor,colortbl}
\usepackage{amsfonts}
\usepackage{makeidx}
\usepackage{amsmath}
\usepackage{accents}
\usepackage{nameref}
\usepackage{amssymb}
\usepackage{bm}
\usepackage{enumerate}
\usepackage{cite}
\usepackage{showkeys}
\usepackage{subfig}
\usepackage[english]{babel,hyperref}
\usepackage{psfrag}
\usepackage{subfig}
\usepackage{color}
\usepackage{float}
\usepackage{amscd}
\usepackage[mathscr]{euscript}
\usepackage{lipsum}
\usepackage{epstopdf}
\usepackage{algorithm}
\usepackage{algpseudocode,caption}
\usepackage{graphicx}
\usepackage{hyperref}
\everymath{\displaystyle}
\newcommand{\N}{\ensuremath{\mathbb{N}}}

\newcommand{\R}{\ensuremath{\mathbb{R}}}
\newcommand{\C}{\ensuremath{\mathbb{C}}}

\renewcommand{\d}{\mathrm{d}}
\newcommand{\di}{\mathbf{d}}

\newcommand{\loc}{\mathrm{loc}}

\newcommand{\p}{\mathbf{p}}
\newcommand{\q}{\mathbf{q}}

\newcommand{\T}{\mathcal{T}}

\renewcommand{\Re}{\mathrm{Re}\,}
\renewcommand{\Im}{\mathrm{Im}\,}

\newcommand{\eps}{\varepsilon}

\renewcommand{\S}{\mathbb{S}}

\renewcommand{\H}{{\mathbf{H}}}
\newcommand{\E}{{\mathbf{E}}}
\newcommand{\Ei}{{\mathbf{E}_\mathrm{in}}}

\renewcommand{\u}{\mathbf{u}}
\renewcommand{\v}{\mathbf{v}}

\usepackage{mathtools}

\newcommand{\D}{{\mathbf{D}}}
\newcommand{\B}{{\mathbf{B}}}

\newcommand{\curl}{\mathrm{curl}\,}

\newcommand{\x}{\mathbf{x}}
\newcommand{\y}{\mathbf{y}}
\newcommand{\h}{\mathbf{h}}
\newcommand{\z}{\mathbf{z}}

\newcommand{\g}{\mathbf{g}}
\newcommand{\f}{\mathbf{f}}
\renewcommand{\u}{\mathbf{u}}
\renewcommand{\v}{\mathbf{v}}

\everymath{\displaystyle}
\newtheorem{defi}{Definition}
\newtheorem{lemma}[defi]{Lemma}
\newtheorem{theorem}[defi]{Theorem}
\newtheorem{corollary}[defi]{Corollary}

\begin{document}
\title{On  orthogonality sampling method for Maxwell's equations \\ and its applications to experimental data}

\author{Thu Le\thanks{Department of Mathematics, University of Wisconsin-Madison, Madison, WI 53706; (\texttt{tle38@wisc.edu})} \and Dinh-Liem Nguyen\thanks{Department of Mathematics, Kansas State University, Manhattan, KS 66506; (\texttt{dlnguyen@ksu.edu})}   
}
\date{}
\maketitle

\begin{abstract}

This paper addresses the inverse scattering problem for Maxwell's equations. We first show that a bianisotropic scatterer can be uniquely determined from multi-static far-field data through the factorization analysis of the far-field operator. Next, we investigate a modified version of the orthogonality sampling method, as proposed in Le \textit{et al} [2022 \textit{Inverse Problems} 38 025007], for the numerical reconstruction of the scatterer. Finally, we apply this sampling method to invert unprocessed 3D experimental data obtained from the Fresnel Institute. Numerical examples with synthetic scattering data for bianisotropic targets are also presented to demonstrate the effectiveness of the method.

\end{abstract}
\sloppy

\bigskip
\section{Introduction}
Inverse scattering problems for Maxwell’s equations arise in many applications, including non-destructive testing, radar, and medical imaging    {\cite{Kong1990, Colton2013}}. In these problems, one seeks to reconstruct certain information about an unknown target (the scatterer) from measurements of the scattered electromagnetic field. In this work, we focus on the inverse scattering problem for Maxwell’s equations in an inhomogeneous bianisotropic medium, using far-field scattering data at a single frequency. We aim to study a numerical method for  determining the location and shape of a 3D scatterer, and to validate its efficiency and robustness through tests on both synthetic and unprocessed experimental data.

Sampling methods, also known as qualitative approaches, are numerical techniques for solving inverse problems and have been extensively studied in recent years. Their key advantages include being non-iterative, computationally efficient, and not requiring a priori information about unknown targets. The linear sampling method \cite{Colton1996} is considered the first sampling method and
has inspired several related approaches, such as the factorization method \cite{Kirsch1998} and the point source method \cite{Potthast1996}. We refer the reader to \cite{Colton2013, Kirsch2008, Cakoni2006} for additional examples of sampling methods.

The orthogonal sampling method (OSM) is a recent development in sampling methods and was first introduced by Potthast in \cite{Potthast2010}. There are many studies using the OSM to solve the inverse problems for the Helmholtz equation (see, e.g., \cite{Ahn2020, Akinci2016, Griesmaier2011, Le2023}) and, more recently, for Maxwell’s equations (see \cite{Harris2020, Le2022, Nguyen2019}). The OSM is also known as the direct sampling method in other works (e.g., \cite{Ito2012, Park2018}). Compared to classical sampling methods, the OSM has a simpler implementation and is more robust against noise in the data, although its theoretical justification is less developed. A modified version of the OSM was proposed in \cite{Le2022} for Maxwell equations in isotropic and anisotropic media. The modified orthogonal sampling method (MOSM) can apply to more types of scattering data compared with the OSM studied in~\cite{Harris2020} and outperforms the OSM in inverting 3D experimental data from the Fresnel Institute~\cite{Geffrin2009}. Our first objective in this work is to extend the MOSM to the case of bianisotropic scattering media. This is achieved through the factorization analysis of the far-field operator, which also establishes the uniqueness of determining the target from multi-static scattering data for the inverse problem. {  The challenges associated with bianisotropic media arise from the fact that the constitutive relations for Maxwell's equations are more complicated than those for standard media. As a result, the second-order PDE formulation of Maxwell's equations includes not only the usual zero- and second-order derivative terms, but also first-order terms. However, due to the smallness of the coefficients representing the bianisotropy, we can treat these first-order terms as perturbations to the standard second-order Maxwell equations.}


Our second objective is to apply the MOSM to invert unprocessed 3D experimental data from the Fresnel Institute~\cite{Geffrin2009}.  A numerical study using processed experimental data from the Fresnel Institute in \cite{Le2022} shows that the MOSM outperforms both its original version and the factorization method. To further validate its stability and efficiency, we will test the MOSM using unprocessed experimental data. Additional results on inverting these processed 3D experimental datasets can be found in \cite{Catap2009, Chaum2009, Zaeyt2009,Li2009, Litma2009}.
Furthermore, the unprocessed 3D Fresnel datasets are highly sensitive to noise and, to the best of our knowledge, not many inversion algorithms have been tested on them before.

The remainder of this paper is organized as follows. In Section \ref{se:problformula}, we present the problem formulation for the inverse scattering 
 problem and an assumption for the well-posedness of the corresponding direct problem. Section \ref{se:unique} provides the factorization analysis of the far-field operator, establishing the uniqueness of the inverse problem and the theoretical foundation for the sampling method. Section \ref{se:sampling} introduces the MOSM, including a resolution analysis of its imaging function and a stability estimate for the method under noisy data. We perform a numerical study to validate the method using synthetic and unprocessed experimental data in Section \ref{se: results}. Finally, Section \ref{se: conclude} concludes the paper.
\section{Problem formulation}
\label{se:problformula}
We consider the scattering of time-harmonic electromagnetic waves by a bianisotropic medium at a fixed frequency $\omega>0$. Suppose that the electric field $\E$ and magnetic field $\H$ satisfy Maxwell's equations
\begin{align}
\label{eq:maxwell}
    \curl  \E-i\omega \B=0, && 
    \curl  \H+i\omega \D=0,\quad \text{in } \mathbb{R}^3,
\end{align}
where $\B$ and $\D$ are the magnetic induction and electric displacement, respectively.  Assume that the medium is characterized by bounded matrix-valued functions $\eps, \mu,\xi,\zeta:\R^3\rightarrow\C^{3\times 3} $, representing the electric permittivity, magnetic permeability, and magnetoelectric coupling tensors.   We refer to \cite{Kong1990} for the following constitutive relations for bianisotropic media
\begin{align}
\label{eq:constitutive}
    \B=\mu \textbf H +\zeta \sqrt{\eps\mu} \textbf E,  &&
    \D=\eps \textbf E + \xi \sqrt{\eps\mu} \textbf H.
\end{align}
Let $\Omega\subset \mathbb{R}^3$ be a bounded Lipschitz domain occupied by the inhomogeneous medium. Assume that $\mathbb{R}^3\setminus \overline{\Omega}$ is connected and that the medium is homogeneous outside $\Omega$, meaning $\eps=\eps_0I_3,\mu=\mu_0I_3$ for some $\eps_0,\mu_0 >0$, and $\xi=\zeta=0$ in $\mathbb{R}^3\backslash \overline{\Omega}$, where $I_3$ is the $3\times 3$ identity matrix. We define  wavenumber $k=\omega \sqrt{\eps_0\mu_0}$, and the relative permittivity and permeability as $\eps_r=\eps/\eps_0$, $\mu_r=\mu/\mu_0.$ 

By rescaling the fields without changing notations, $\textbf E=\sqrt{\eps_0} \textbf E$ and $\textbf H=\sqrt{\mu_0} \textbf H$, 
we derive from the system \eqref{eq:maxwell}--\eqref{eq:constitutive} the following equation for the total electric field $\E$
\begin{align}
\label{eq:order2total}
   \curl[\mu_r^{-1} \curl \E] +ik[\xi  \mu_r^{-1}\curl  \E - \curl (\mu_r^{-1}\zeta \E)]-k^2(\eps_r-\xi \mu_r^{-1}\zeta)\E=0,\quad \text{in } \mathbb{R}^3.
\end{align}
Suppose that the transmission conditions across the boundary $\partial \Omega$ of $\Omega$ are given by
\begin{align}
\label{eq:transmission}
     \nu \times \textbf E_+= \nu \times \textbf E_-,  &&
     \nu \times \curl  \textbf E_+ = \nu \times [\mu_r^{-1}(\curl \textbf E_- - ik\xi \textbf E_-)],
\end{align}
where $\E_+$ and $\E_-$ are the traces on $\partial \Omega$ from the exterior and interior of 
$\Omega$ for $\E$, respectively, and $\nu$ is  the outward unit normal vector on $\partial \Omega$.

Assume that the inhomogeneous medium is illuminated by an incident  field $\Ei$ satisfying
\begin{align}
    \curl \curl \Ei - k^2\Ei=0,\quad \text{in } \mathbb{R}^3.
    \label{homoein}
\end{align}
Let $\u = \textbf {E}-\Ei$ be the scattered electric field. It follows from \eqref{eq:order2total} and \eqref{homoein} that $\u$ satisfies 
\begin{align}
\label{eq:order2u}
   & \curl[\mu_r^{-1} \curl  \u] +ik[\xi  \mu_r^{-1}\curl  \u - \curl (\mu_r^{-1}\zeta \u)]-k^2(\eps_r-\xi \mu_r^{-1}\zeta)\u\nonumber\\
   &= \curl[Q\curl  \Ei] -ik[\xi \mu_r^{-1}\curl \Ei -  \curl(\mu_r^{-1}\zeta \Ei)]+k^2(P-\xi \mu_r^{-1}\zeta)\Ei, \quad \text { in } \mathbb{R}^3,
\end{align}
where $$P:=\eps_r-I_3, \quad Q:=I_3-\mu_r^{-1},$$ with support in $\Omega$.
Furthermore,  the corresponding transmission conditions for $\u$ follow \eqref{eq:transmission}. The scattered field $\u$ is also required to satisfy the  Silver–Müller radiation condition 
\begin{align}
\label{eq:radiation}
 \curl \u \times \frac{\x}{|\x|} - ik \u = \mathcal{O}(|\x|^{-2}) \quad \text{as } |\x| \to \infty,  
\end{align}
which holds uniformly with respect to $\x/|\x|$. {We  refer to Figure \ref{fig:schematic} for an illustration of the geometry of the scattering problem.}

For any open and connected domain $D\subset \R^3$  with a Lipschitz boundary, let us denote
\begin{align*}
   &H(\curl,D)=\{\v\in [L^2(D)]^3: \curl \v  \in [L^2(D)]^3\},\\
   &H_{\text{loc}}(\curl,\R^3)=\{\v: \R^3\rightarrow \C^3:   {\v}|_B \in H(\curl, B) \text{ for all balls } B \subset \R^3\}.
\end{align*}
Throughout this paper, we also denote by 
$\langle \cdot , \cdot\rangle$ and 
$\|\cdot\|$ the inner product and norm on $[L^2(\Omega)]^3\times [L^2(\Omega)]^3$. In addition,  $H(\curl , D)$ is equipped by the usual inner product
$\langle \cdot , \cdot\rangle_{H(\curl , D)} = \langle \cdot , \cdot\rangle + \langle \curl \cdot , \curl \cdot\rangle.$
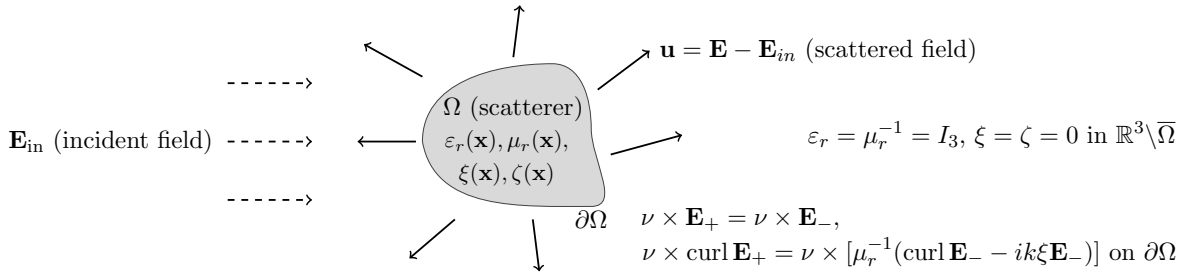
\begin{figure}[ht]
\centering
\resizebox{0.9\linewidth}{!}{%
\begin{tikzpicture}
\filldraw[fill=gray!30, draw=gray!50!black] 
(0,0) to[out=90,in=180] (2.,1.2) to[out=0,in=90] (2.6,0.2)
to[out=-90,in=0] (2.6,-1.) to[out=180,in=-80] (0,0);
\node[] at (1.4,0.5) {$\Omega$ (scatterer)};
\node[] at (2.6,-1.2) {$\partial \Omega$};
\node[] at (4.9,-1.2) { $\nu \times \textbf E_+= \nu \times \textbf E_-$,};
\node[] at (7.5,-1.7) { $ \nu \times \curl  \textbf E_+ = \nu \times [\mu_r^{-1}(\curl \textbf E_- - ik\xi \textbf E_-)]$ on  $\partial \Omega$};

\node at (1.3,0.) {\small $\varepsilon_r(\x), \mu_r(\x)$,};
\node at (1.3,-0.5) { \small $\xi(\x), \zeta(\x)$};
\draw[->, thick, dashed] (-3,0.9) -- (-1.7,0.9);
\draw[->, thick, dashed] (-3,0) -- (-1.7,0);
\draw[->, thick, dashed] (-3,-0.9) -- (-1.7,-0.9);
\node[] at (-4.8,0) { $\mathbf{E}_\text{in}$  (incident field)};
\draw[->, thick] (1.4,1.3) -- (1.5,2.1);
\draw[->, thick] (-0.1,0) -- (-1,0);
\draw[->, thick ] (2.7,0.8) -- (3.5,1.4);
\draw[->, thick] (2.9,-0.2) -- (4.,0.1);
\draw[->, thick] (1.7,-1.2) -- (1.8,-2);
\draw[->, thick] (0,1) -- (-0.9,1.5);
\draw[->, thick] (0.5,-1.2) -- (-0.2,-1.8);
\node[] at (6.1,1.4) { $\mathbf{u}=\E-\E_{in}  $ (scattered field)};
\node[right] at (5.8,0.1) {   $\varepsilon_r = \mu_r^{-1} = I_3$, $\xi = \zeta = 0$ in $\mathbb{R}^3\backslash \overline{\Omega}$};
\end{tikzpicture}%
}
\caption{  {A diagram illustrating the geometry of the scattering problem.}}
\label{fig:schematic}
\end{figure}

{ Given the coefficients  $\eps_r,\mu_r, \mu_r^{-1}, \xi,\zeta$, and the incident field $\Ei$ the direct problem is to determine the scattered $\u$ satisfying \eqref{eq:order2u}-\eqref{eq:radiation}. } 
Under the Assumption I below,  we refer to \cite{Nguyen2019} for the existence and uniqueness of the solution  $\u\in H_{\loc}(\curl,\mathbb{R}^3)$ to the direct problem \eqref{eq:order2u}-\eqref{eq:radiation}.
\\~\\
{\textbf{Assumption I.}} Assume that $\eps_r,\mu_r, \mu_r^{-1}, \xi,\zeta \in [L^\infty(\Omega)]^{3\times3}$ are symmetric almost everywhere. Furthermore, suppose that there exist positive constants $c_1, c_2, \alpha,\beta$ such that for any $a \in \mathbb{C}^3$, the following inequalities hold almost everywhere in $\Omega$
\begin{align}
&\Re(\mu_r^{-1}a\cdot \overline{a})\ge c_1|a|^2,\quad \quad  \quad \Re((\eps_r-\xi \mu_r^{-1}\zeta)a\cdot \overline{a})\ge c_2|a|^2,
\\
&-\Im(\mu_r^{-1}a\cdot\overline{a})\ge \alpha|a|^2,\quad \quad \Im((\eps_r-\xi \mu_r^{-1}\zeta) a\cdot \overline{a})\ge \beta|a|^2, \label{imcond}
\\
&\||\mu_r^{-1}\xi|_F\|^2_ {L^\infty}+\||\mu_r^{-1}\zeta|_F\|^2_ {L^\infty} < 2\min\{c_1c_2,\alpha\beta\},\label{imcond1}
\end{align}
where $|\cdot|_F$ is the Frobenius matrix norm. We denote 
\begin{align*}
\mathbb{S}^2 = \{\widehat \x \in \R^3: |\widehat \x| = 1 \}, \quad \quad \mathbf{L}^2_t(\mathbb{S}^2) = \{ \v \in [L^2(\mathbb{S}^2)]^3: \widehat{\x} \cdot \v(\widehat{\x}) = 0,\,\, \widehat{\x} \in \mathbb{S}^2 \},
\end{align*}
and consider the incident plane wave 
\begin{align}
    \Ei(\x,\di, \q)=\q (\di)e^{ik\x\cdot\di},\label{eq:incidentwave}
\end{align}
 where  $\di \in \S^2$ is the direction of the incident propagation 
and $\q \in \R^3$ is the polarization vector such that $\q \cdot \di =0$.  It is well-known that the scattered wave $\u(\x,\di,\q)$ to \eqref{eq:order2u}-\eqref{eq:radiation} has the asymptotic behavior
\begin{align}
\label{eq:farfield}
\u(\x,\di,\q) = \frac{e^{ik|\x|}}{ |\x|} \left ( \u^\infty(\widehat{\x},\di,\q) + \mathcal{O}\left( \frac{1}{|\x|} \right) \right) \quad \text{ as } \quad |\x| \to \infty,
\end{align}
where $\u^\infty(\widehat{\x},\di,\q)$ is called the far-field pattern of $\u$.
We conclude this section by stating the inverse problem of interest.
\\

\textbf{Inverse problem:} Given measurements of $\u^\infty(\widehat{\x},\di,\q)$ for all $\widehat{\x},\di\in \mathbb{S}^2$ at a fixed wavenumber $k>0$,  determine the domain $\Omega$ representing location and shape of the scatterer.
\section{Unique determination and factorization analysis}
\label{se:unique}
In this section, we study the factorization analysis that stems from the factorization method by Kirsch~\cite{Kirsch2008}. This analysis helps establish the uniqueness of our inverse  problem and provides a theoretical foundation for the sampling method discussed in Section \ref{se:sampling}.   

We begin by introducing  the far-field operator 
  $\mathcal{F}: \mathbf{L}^2_t(\S^2) \to  \mathbf{L}^2_t(\S^2)$  and
  the solution operator  $\mathcal{G}: [L^2(\Omega)]^3\times [L^2(\Omega)]^3 \to \mathbf L^2_t(\mathbb{S}^2) $  as follows
\begin{align}
\label{farfield}
(\mathcal{F}\g)(\widehat{\x})  = \int_{\S^2}  \u^\infty(\widehat{\x},\di, \g(\di)) \d s(\di), \quad \quad \mathcal G(\f,\g)=\u^\infty.
\end{align}
Here, $\u^\infty$ is the far-field pattern of $\u\in H_{\loc}(\curl,\mathbb{R}^3)$ to \eqref{eq:order2u}-\eqref{eq:radiation}  satisfying 
    \begin{align}
 \int_{\mathbb{R}^3} (\text{curl }  \textbf u \cdot  \text{curl } \overline  {\mathbf v}- k^2  \textbf  u\cdot \overline {\mathbf v})\d\x
 =&\int_{\Omega}{[k^2(P-{\xi}\mu_r^{-1}\zeta)(\textbf u +\f)-ik{\xi}\mu_r^{-1}(\text{curl } \textbf u+\g)}]\cdot \overline{{\mathbf v}} \nonumber
 \\&\quad +{[ik\mu_r^{-1}\zeta (\textbf u+\f) + Q(\text{curl }\textbf u+\g) ]}\cdot \text{curl } \overline{{\mathbf v}}\d\x,\label{weak}
\end{align}
with  $\f= \Ei$, $ \g = \curl \Ei$, and for all ${\mathbf v} \in H(\text{curl},\mathbb{R}^3)$ with compact support. It is known in \cite{Kirsch2007} that the variational form \eqref{weak} is equivalent
to the integro-differential equation
\begin{align}
\u (\x)& =(k^2+\nabla \text{div}) \int_\Omega[(P-{\xi}\mu_r^{-1}\zeta)(\textbf u+\f)-\frac{i}{k}{\xi}\mu_r^{-1}(\text{curl } \textbf u+\g)]\Phi(\x,\y)\d \y\nonumber
\\&\quad \quad \quad +\text{curl} \int_\Omega[ik\mu_r^{-1}\zeta (\textbf u+\f)+Q(\text{curl}\textbf u +\g)]\Phi(\x,\y)\d\y,
\end{align}
where $\Phi(\x,\y)={e^{ik|\x-\y|}}/(4\pi|\x-\y|)$ is the free-space Green function of the scalar
Helmholtz equation.
Next, we define two auxiliary operators 
 $\mathcal{H}_{1,2}: \mathbf L^2_t(\mathbb{S}^2) \to [L^2(\Omega)]^3$ as
$$(\mathcal{H}_1\g)(\x)= \int_{\S^2} \g(\di) e^{ik\x\cdot\di} \d s(\di), \quad \quad \mathcal{H}_2\g=\curl \mathcal{H}_1\g,\quad  \x \in \Omega.$$
Then, the Herglotz operator is given by $\mathcal{H}: \mathbf L_t^2(\mathbb{S}^2)\to [L^2(\Omega)]^3\times [L^2(\Omega)]^3 $ and 
$\mathcal H\g=(\mathcal H_1\g,\mathcal H_2\g).$
It is well-known that the  Herglotz operator $\mathcal{H}$ is compact and injective. In addition, 
the adjoint operators of $\mathcal H_{1,2}$ are expressed as $\mathcal H^*_{1,2}:  [L^2(\Omega)]^3\to \mathbf L_t^2(\mathbb{S}^2) $ 
$$(\mathcal H^*_1\f)(\di)=\di\times \left(\int_\Omega \f(\x)e^{-ik\di\cdot \x} 
\d\x \right)\times \di, \quad \quad(\mathcal  H^*_2\f)(\di)=ik\di \times \int_\Omega \f(\x)e^{-ik\di\cdot \x} \d\x,\quad \di\in \mathbb{S}^2.$$
\\
The adjoint operator  of  $\mathcal H$ is then represented as $\mathcal H^*: [L^2(\Omega)]^3\times [L^2(\Omega)]^3 \to \mathbf L_t^2(\mathbb{S}^2) $ 
$$\mathcal H^*(\f,\g)=\mathcal H^*_1\f+\mathcal H^*_2\g.$$
It is clear that 
the operators defined above are all linear and bounded under Assumption I. Now, we define two operators $\mathcal T_{1,2}: [L^2(\Omega)]^3\times [L^2(\Omega)]^3\to [L^2(\Omega)]^3$ as
\begin{align}
  &  \mathcal \T_1(\f,\g)=k^2(P-\xi \mu_r^{-1}\zeta)(\u+\f)-ik\xi \mu_r^{-1}(\curl  \u + \g),\notag \\
  &  \mathcal \T_2(\f,\g)=ik\mu_r^{-1}\zeta(\u+\f)+Q(\curl  \u +\g).
    \label{defT12}
\end{align}
Then the middle operator  $\mathcal T: [L^2(\Omega)]^3\times [L^2(\Omega)]^3\to [L^2(\Omega)]^3\times [L^2(\Omega)]^3$ is given by
\begin{align}
    \mathcal T=(\mathcal \T_1,\mathcal \T_2).\label{defT}
\end{align}
Similarly as in~\cite{Le2022} we can show that $\mathcal G=\mathcal H^*\mathcal T$ and obtain the following factorization 
\begin{align}
    \label{fac}
\mathcal{F}=\mathcal H^*\mathcal T\mathcal H.
\end{align}
\begin{lemma}
   Suppose that Assumption I holds, then the middle operator $\T$ defined in \eqref{defT} is linear and bounded on $[L^2(\Omega)]^3\times [L^2(\Omega)]^3$. 
   \label{T:linear}
\end{lemma}
\begin{proof} 
Let $(\f, \g),(\widetilde{\f}, \widetilde{\g})\in [L^2(\Omega)]^3 \times [L^2(\Omega)]^3$, and let $\u, \widetilde{\u} \in H(\curl,\Omega)$ be the corresponding solutions to \eqref{weak} for the data $(\f, \g)$ and $(\widetilde{\f}, \widetilde{\g})$, respectively. For any $ a,b \in \mathbb{C}^3$, by definition of $\T$ in \eqref{defT} 
$$a\T(\f,\g)+b\T(\widetilde\f,\widetilde \g)= (a\T_1(\f,\g)+b\T_1(\widetilde\f,\widetilde \g),a\T_2(\f,\g)+b\T_2(\widetilde\f,\widetilde \g)).$$
Using the linearity and well-posedness of  problem \eqref{weak}, the first component can be rewritten as
\begin{align*}
   a\T_1(\f,\g)+b\T_1(\widetilde\f,\widetilde \g)
   &= k^2(P-\xi \mu_r^{-1}\zeta)(a\u+b\widetilde \u+a\f+b\widetilde\f)-ik\xi \mu_r^{-1}(\curl(a\u+b\widetilde \u)  + a\g+b\widetilde\g)\\
   &=\T_1(a\f+b\widetilde\f,a\g+b\widetilde \g)=\T_1(a(\f,\g)+b(\widetilde\f,\widetilde \g)).
\end{align*} 
Similarly, for the second component,
$a\T_2(\f,\g)+b\T_2(\widetilde\f,\widetilde \g)=\T_2(a(\f,\g)+b(\widetilde\f,\widetilde \g)).$
Together, these results yield the linearity of $\T$.
Next, we can derive from \eqref{defT12} that 
\begin{align*}
    &\|\T_1(\f,\g)\|\le \max \{ k^2 \||P-\xi \mu_r^{-1}\zeta|_F\|_{L^\infty}, k\| |\xi \mu_r^{-1}|_F\|_{L^\infty}\}(\|\u\|_{H(\curl,\Omega)}+\|(\f,\g)\|),\\
    &\|\T_2(\f,\g)\|\le \max\{ k\||\mu_r^{-1}\zeta|_F\|_{L^\infty}, \||Q|_F\|_{L^\infty} \}(\|\u\|_{H(\curl,\Omega)}+\|(\f,\g)\|).
\end{align*}
Moreover, it follows from the well-posedness of the direct problem  $\|\u\|_{H( \curl,\Omega)}\le c\|(\f,\g)\|,$ for some constant $c>0$. 
 This implies that $\T$ is bounded.
\end{proof}
The following theorem is helpful for the factorization analysis.
\begin{theorem}
\label{imT}
Under Assumption I, there exists $\gamma >0$ such that 
\begin{align}
   \langle\Im\mathcal T(\f,\g),(\f,\g)\rangle \ge \gamma \left\|(\f,\g)\right\|^2 \label{coerciveT}
\end{align}
for all $(\f,\g)\in \text{\text{Range}}(\mathcal H) \subset [L^2(\Omega)]^3\times [L^2(\Omega)]^3 $.
\end{theorem}
\begin{proof}
 Let $(\f,\g)\in \text{Range} (\mathcal H)$ and $\u$ be the corresponding solution to \eqref{weak} with this data. Denote $\textbf{w}_1=\f+\u, \textbf{w}_2=\g+\curl \u$. Then, we have 
$$
    \langle\mathcal T(\f,\g),(\f,\g)\rangle = I_1-I_2,
$$
where
\begin{align*}
    I_1&=\int_\Omega k^2(P-\xi \mu_r^{-1}\zeta)\textbf{w}_1\cdot \overline{\textbf{w}}_1-ik\xi \mu_r^{-1}{\textbf{w}}_2\cdot \overline{\textbf{w}}_1+ik\mu_r^{-1}\zeta \textbf{w}_1\cdot \overline{\textbf{w}}_2+Q\textbf{w}_2\cdot \overline{\textbf{w}}_2
    \d \x,\\
    I_2&= \int_\Omega   \left(k^2(P-\xi \mu_r^{-1}\zeta)\textbf{w}_1-ik\xi \mu_r^{-1}\textbf{w}_2\right) \cdot\overline{\u}+ \left(ik\mu_r^{-1}\zeta\textbf{w}_1+Q\textbf{w}_2\right)\cdot \curl \overline{\u}
    \d \x.
\end{align*}
The symmetry of $\mu_r^{-1}$ and $ \xi $ yields  that $\xi \mu_r^{-1}{\textbf{w}}_2\cdot \overline{\textbf{w}}_1={\mu_r^{-1}\xi \overline{\textbf{w}}_1\cdot {\textbf{w}}_2}.$
Now, taking the imaginary part of both sides of $I_1$ and using Assumption I, we obtain
\begin{align}
    \Im(I_1)\ge \int_\Omega k^2\beta|\textbf{w}_1|^2-k\Re(\mu_r^{-1}\xi\overline{\textbf{w}}_1\cdot {\textbf{w}}_2)+k\Re(\mu_r^{-1}\zeta \textbf{w}_1\cdot \overline{\textbf{w}}_2)+\alpha|{\textbf{w}}_2|^2 \d \x.\label{rightim1}
\end{align}
Next, substitute the following identities into the right-hand side in \eqref{rightim1}
\begin{align*}
    &\frac{k^2|\mu_r^{-1}\xi \overline{\textbf{w}}_1|^2}{2\alpha}-k\Re(\mu_r^{-1}\xi \overline{\textbf{w}}_1\cdot {\textbf{w}}_2) + \frac{\alpha|{\textbf{w}}_2|^2}{2}=\frac{1}{2}\left|\frac{k\mu_r^{-1}\xi \overline{\textbf{w}}_1}{\sqrt{\alpha}}-\sqrt{\alpha}\overline{\textbf{w}}_2\right|^2,\\
    &\frac{k^2|\mu_r^{-1}\zeta \textbf{w}_1|^2}{2\alpha}+k\Re(\mu_r^{-1}\zeta \textbf{w}_1\cdot \overline{\textbf{w}}_2)+\frac{\alpha|{\textbf{w}}_2|^2}{2}=\frac{1}{2}\left|\frac{k\mu_r^{-1}\zeta \textbf{w}_1}{\sqrt{\alpha}}+{\sqrt{\alpha}}{\textbf{w}_2}\right|^2,
\end{align*}
then we get
\begin{align*}
    \Im(I_1)\ge &{\frac{1}{2}\left\|\frac{k\mu_r^{-1}\xi \overline{\textbf{w}}_1}{\sqrt{\alpha}}-\sqrt{\alpha}\overline{\textbf{w}}_2\right\|^2
    +\frac{1}{2}\left\|\frac{k\mu_r^{-1}\zeta \textbf{w}_1}{\sqrt{\alpha}}+{\sqrt{\alpha}}{\textbf{w}_2}\right\|^2} 
    +k^2\int_\Omega \beta|\textbf{w}_1|^2-\frac{|\mu_r^{-1}\xi \overline{\textbf{w}}_1|^2}{2\alpha}-\frac{|\mu_r^{-1}\zeta \textbf{w}_1|^2}{2\alpha}\d \x.
\end{align*}
Moreover, notice that $$\int_\Omega \beta|\textbf{w}_1|^2-\frac{|\mu_r^{-1}\xi \overline{\textbf{w}}_1|^2}{2\alpha}-\frac{|\mu_r^{-1}\zeta \textbf{w}_1|^2}{2\alpha} \d \x \ge C\|\textbf{w}_1\|^2,$$
where $$C:=\beta-\frac{1}{2\alpha}\||\mu_r^{-1}\xi|_F\|^2_{L^\infty}-\frac{1}{2\alpha}\||\mu_r^{-1}\zeta|_F\|^2_{L^\infty},$$ which is positive by Assumption I. On the other hand, since $\u$ is the solution to \eqref{weak}, we have 
    \begin{align*}
 \int_{\mathbb{R}^3} (\text{curl }  \textbf u \cdot  \text{curl } \overline  {\mathbf v}- k^2  \textbf  u\cdot \overline {\mathbf v})\d\x
 =&\int_{\Omega}{[k^2(P-{\xi}\mu_r^{-1}\zeta)\textbf{w}_1-ik{\xi}\mu_r^{-1}\textbf{w}_2}]\cdot \overline{{\mathbf v}} +{[ik\mu_r^{-1}\zeta \textbf{w}_1 + Q\textbf{w}_2]}\cdot \text{curl } \overline{{\mathbf v}}\d\x,
\end{align*}
for all ${\mathbf v} \in H(\text{curl},\mathbb{R}^3)$ with compact support. Let $r>0$ be a constant such that $\overline{\Omega}\subset \{\x\in \R^3:|\x|<r\}$. Next, plug ${\mathbf v}=\phi \u$ into the equation right above,  where $\phi\in \C^\infty(\R^3)$ is a cut-off function with $\phi=1$ for $|\x|<r$ and $\phi=0$ for $|\x|\geq 2r$. We then obtain
    \begin{align}
&I_2=\int_{\Omega}{[k^2(P-{\xi}\mu_r^{-1}\zeta)\textbf{w}_1-ik{\xi}\mu_r^{-1}\textbf{w}_2}]\cdot \overline{{\mathbf u}} +{[ik\mu_r^{-1}\zeta \textbf{w}_1 + Q\textbf{w}_2]}\cdot \text{curl } \overline{{\mathbf u}}\d\x \notag \\
&= \int_{|\x|<r} |\text{curl }  \u|^2  - k^2   |\u|^2\d\x+\int_{r<|\x|<2r} \text{curl }  \textbf u \cdot  \text{curl } (\phi \overline  {\mathbf u})- k^2   |\u|^2\phi\d\x.\label{I2stoke}
\end{align}
Using Stokes’ theorem and taking the imaginary on both sides of \eqref{I2stoke} yields
\begin{align}
     \Im(I_2)=\Im \int_{|\x|=r}(\widehat{\x} \times \curl \u) \cdot \overline{\u} \d s(\x).\label{imi22}
\end{align}
Let $|r|\rightarrow \infty$ in \eqref{imi22}, then by the radiation condition \eqref{eq:radiation} and the asymptotic behavior of $\u^\infty$ in \eqref{eq:farfield} 
\begin{align}
    \Im(I_2)\nonumber=\frac{k}{(4\pi)^2}\int_{\mathbb{S}^2}|\u^\infty|^2\d s(\x).
\end{align}
Combining the estimations above for $\Im(I_1)$ and $\Im(I_2)$, we derive
\begin{align}
    &\Im\langle\mathcal T(\f,\g),(\f,\g)\rangle =\Im(I_1)-\Im(I_2)\nonumber\\
    &\geq {\frac{1}{2}\left\|\frac{k\mu_r^{-1}\xi \overline{\textbf{w}}_1}{\sqrt{\alpha}}-\sqrt{\alpha}\overline{\textbf{w}}_2\right\|^2
    +\frac{1}{2}\left\|\frac{k\mu_r^{-1}\zeta \textbf{w}_1}{\sqrt{\alpha}}+{\sqrt{\alpha}}{\textbf{w}_2}\right\|^2} +k^2C\|\textbf{w}_1\|^2 +\frac{k}{(4\pi)^2}\int_{\mathbb{S}^2}|\u^\infty|^2\d s(\x).\label{eq:imT}
\end{align}
Thanks to Lemma \ref{T:linear} and the fact that $$\Im \T =  \frac{\T-\T^*}{2i},$$ we have $\Im\langle\mathcal T(\f,\g),(\f,\g)\rangle=\langle\mathcal \Im \mathcal T(\f,\g),(\f,\g)\rangle$.
Now,  assume that there is no $\gamma >0$ such that \eqref{coerciveT} satisfies. We then can construct a sequence $\{(\f_n,\g_n)\}_{n=1}^\infty$ in $L^2(\Omega)^3 \times L^2(\Omega)^3$ such that $\|(\f_n,\g_n)\|=1$ for all $n\in \N$, and 
\begin{align}
    \lim_{n\rightarrow \infty}\langle \Im \T(\f_n,\g_n),(\f_n,\g_n)\rangle =0.
\label{limT}
\end{align}Let $\u_n$ be the corresponding solution to \eqref{weak} with data $(\f_n,\g_n)$ for each $n$. 
It follows from \eqref{eq:imT} and \eqref{limT} that 
\begin{align}
    {\bf{w}}_{1,n} = \f_n+\u_n \xrightarrow{n \to \infty} 0 \quad \text{ and } \quad {\bf{w}}_{2,n} = \g_n+\curl \u_n \xrightarrow{n \to \infty} 0,\label{w}
\end{align}
in $[L^2(\Omega)]^3$. 
Together with the following boundedness of the integro-differential operators in \eqref{weak} (see \cite{Kirsch2007})
$$\|\u\|_{H(\text{curl},\Omega)}\leq C_1\|\f+\u\|+C_2\| \g+\curl \u\|,$$
for some constants $C_1, C_2>0$, which implies $\u_n  \xrightarrow{n \to \infty} 0$ in $H(\text{curl},\Omega)$. Then \eqref{w} yields $(\f_n,\g_n) \xrightarrow{n \to \infty} 0$ in $L^2(\Omega)^3 \times L^2(\Omega)^3$, which leads to a contradiction whence $\|(\f_n,\g_n)\|=1$. 
\end{proof}
The next lemma provides a characterization of the scatterer  $\Omega$.
\begin{lemma} 
Let $\phi_{\y_s}(\di)=(\di\times \p)\times \di e^{-ik\di \cdot \y_s}$, then $\y_s\in \Omega$  if and only if $\phi_{\y_s}\in \text{Range}(\mathcal{H}^*)$.\label{charac1}
\end{lemma}
\begin{proof} 
 Suppose $\y_s\in \Omega$. From \cite{Kirsch2008}, there exists  $\widetilde \varphi \in L^2(\Omega)$ such that  $$e^{-ik\di \cdot \y_s}=\int_{\Omega} \widetilde \varphi (\y)e^{-ik\di \cdot \y}\d\y,\quad \text{ for all } \di\in \mathbb{S}^2.$$ Substituting this into $\phi_{\y_s}(\di)$, we have \begin{align*} \phi_{\y_s}(\di) =\di \times \int_\Omega \p \widetilde\varphi(\y)e^{-ik\di \cdot \y}\d \y\times \di=(\mathcal H_1^* (\p \widetilde\varphi))(\di)+(\mathcal H_2^*\mathbf{0})(\di) =(\mathcal H^*\g)(\di), \end{align*} where  $\g=(\p \widetilde\varphi,\mathbf{0}) $. Hence, $\phi_{\y_s}\in \text{Range}(\mathcal H^*)$. We next prove the reverse direction by contradiction. Assume that $\y_s\notin \Omega$ and there exists $(\f,\g)\in [L^2(\Omega)]^3\times [L^2(\Omega)]^3$ such that $\phi_{\y_s}=\mathcal{H}^*(\f,\g)$. Notice that from its definition, $\mathcal{H}^*(\f,\g)$  is the far-field pattern of the radiating solution $\v$ (refer to\cite{Kirsch2008})  \begin{align*} \v =(k^2+\nabla \text{div}) \int_\Omega \f(\y)\Phi(\cdot,\y)\d \y +\text{curl} \int_\Omega\g(\y)\Phi(\cdot ,\y)\d\y. \end{align*} Moreover, $\phi_{\y_s}$ is the far-field pattern of $4\pi\text{curl curl}_{\y_s}[\p \Phi(\cdot,\y_s)]/k^2$, where again $\Phi$ is the Green function. By Rellich's Lemma and analytic continuation, we obtain $$\v=\frac{4\pi}{k^2}\text{curl curl }_{\y_s}[\p \Phi(\cdot,\y_s)]\text{ on } \mathbb{R}^3\backslash (\overline{\Omega}\cup \{\y_s\})$$ This leads to a contradiction since the right-hand side contains a singularity at $\y_s$, while $\v$ is analytic outside $\overline{\Omega}$.  \end{proof}

 The following theorem characterizes $\Omega$ based on the factorization of the far-field operator and establishes the uniqueness of the inverse scattering problem.
 \begin{theorem}
     If Assumption I is satisfied, then  domain $\Omega$ can be uniquely
determined by the far-field data $\u^\infty(\widehat{\x},\di,\q)$ in the inverse problem.
 \end{theorem}
\begin{proof}
It is easy to see that $\Im(\mathcal{F})$ is a compact self-adjoint operator.  Moreover, from~\eqref{fac}, it admits the factorization $$\Im(\mathcal{F})=\mathcal H^*\Im(\mathcal T)\mathcal H.$$
By Theorem \ref{imT}, $\Im(\mathcal T)$ is coercive on $\text{Range}(\mathcal H)$. Therefore, applying Corollary 1.22 in \cite{Kirsch2008}, it follows that 
$$
\text{Range}(\mathcal H^*) = \text{Range}(\Im(\mathcal{F})^{1/2}).$$ 
Combining  this with Lemma \ref{charac1}, we have
\begin{align}
\label{charact2}
\y_s \in \Omega \Longleftrightarrow \phi_{\y_s} \in \text{Range}(\Im(\mathcal{F})^{1/2}).
\end{align} 
Now, assume that $\Omega_1$ and $\Omega_2$ are two domains obtained from far-field data $\u_1^\infty$ and $\u_2^\infty$ in the inverse problem, respectively. Suppose that $\u_1^\infty = \u_2^\infty$. Then, the corresponding far-field operators $\mathcal{F}_1$ and $\mathcal{F}_2$, defined in~\eqref{farfield}, are the same. This implies that $$\text{Range}(\Im(\mathcal{F}_1)^{1/2})= \text{Range}(\Im(\mathcal{F}_2)^{1/2}).$$
It follows immediately that $\Omega_1=\Omega_2$, thanks to the complete characterization \eqref{charact2}.
\end{proof} 

\section{Modified orthogonality sampling method}
\label{se:sampling}
In this section, we extend the MOSM in \cite{Le2022} for solving the inverse problem in anisotropic media to bianisotropic media. Let $\y_s\in \R^3$ denote a sampling point and $\p\in \R^3$ be a fixed nonzero vector. We consider the following imaging function 
\begin{align}
\label{imagfuncmosm}
    \mathcal{I}_{MOSM}(\y_s)=\int_{\mathbb{S}^2}\left| \int_{\mathbb{S}^2}\u^{\infty}(\widehat \x,\di, \h(\di) )e^{-ik\di \cdot \y_s}\d s(\di)\right|^2\d s(\widehat \x),
\end{align}
where \begin{align}
    \h(\di)=\alpha_1\di\times \p+\alpha_2(\di\times \p)\times \di,\label{eq:h}
\end{align}
with any  $\alpha_1,\alpha_2\in \C$ such that $|\alpha_1|^2+|\alpha_2|^2>0$.

To analyze the behavior of the imaging function, we connect $\mathcal{I}_{MOSM}$ with the far-field operator $\mathcal{F}$. 
It is well-known that the far-field pattern $\u^\infty(\widehat\x, \di, \q)$ is linear in polarization $\q\in \mathbb{R}^3$ where $\q\cdot \di=0$  and can be written as
$$\u^\infty(\widehat \x,\di,\q) = \u^\infty(\widehat \x,\di)\q,\quad \text{ for all } \widehat \x,\di\in\mathbb{S}^2.$$ Now, let $\psi_{\y_s}(\di) =  \h(\di)e^{-ik\di\cdot \y_s}$. The imaging function in \eqref{imagfuncmosm} satisfies
\begin{align}
     \mathcal{I}_{MOSM}(\y_s)=\int_{\mathbb{S}^2}\left| \int_{\mathbb{S}^2}\u^{\infty}(\widehat \x,\di, \psi_{\y_s}(\di) )\d s(\di)
\right|^2\d s(\widehat \x)=\|\mathcal{F}\psi_{\y_s}\|^2.
\end{align}

The next theorem provides a resolution analysis for the imaging function $\mathcal{I}_{MOSM}$.
\begin{theorem} For any $\y_s\in\R^3$, the imaging function satisfies 
    $$\begin{cases}
&\mathcal{I}_{MOSM}(\y_s)\leq ||\mathcal{G}||^2[(|\alpha_1|^2+k^2|\alpha_2|^2)||W_{\y_s}||^2+(k^2|\alpha_1|^2+|\alpha_2|^2)||V_{\y_s}||^2], \thinspace  \\
&\mathcal{I}_{MOSM}(\y_s)\geq\frac{ \gamma^2}{||\h||^2}[(|\alpha_1|^2+k^2|\alpha_2|^2)||W_{\y_s}||^2+(k^2|\alpha_1|^2+|\alpha_2|^2)||V_{\y_s}||^2]^2>0,  
\end{cases}$$   
where $\h$ is given in \eqref{eq:h} and $\gamma$ is the positive coercive constant in Lemma \ref{imT},
$$W_{\y_s}(\x) := \widetilde W (\y_s - \x),\quad  V_{\y_s}(\x) := \widetilde V (\y_s - \x),\quad  \text{ for }\x \in\Omega,$$
with
\begin{align*}
&\widetilde W (\z) = 4\pi i(\z \times  \p) \frac{\cos(k|\z|) - j_0(k|\z|)}{k|\z|^2 },\quad \text{ where } \quad j_0(x)=\frac{\sin(x)}{x},
\\
 &\widetilde V (\z)=4\pi\frac{ \p|\z|^2 - (\p \cdot \z)\z}{
|\z|^2} j_0(k|\z|) - 12\pi \frac{(\p \cdot \z)\z - \frac{1}
{3}\p|\z|^
2}{
k^2|\z|^
4} (\cos(k|\z|) - j_0(k|\z|)).
\end{align*}
\end{theorem}
\begin{proof}
Using  Theorem~\ref{imT} the proof can be done similarly as in Theorem 14 of \cite{Le2022}, thus we omit it here.
\end{proof}
It is known that the functions $\|W_{\y_s}\|^2$ and $\|V_{\y_s}\|^2$ peak when $|\y_s-\x| < r$ for a small $r>0$, and decay rapidly as $\y_s$ is away from $\x\in \Omega$   {(see \cite{Le2022})}.
The imaging function then has the following behavior.
\begin{corollary}
    For every $\y_s\in \Omega$, the imaging function $\mathcal{I}_{OSM}(\y_s)$ is bounded from below by a positive constant. Moreover, for every $\y_s\notin \Omega$, the imaging function satisfies
$$\mathcal{I}_{OSM}(\y_s)=\mathcal{O}\left(\frac{1}{\text{dist}(\y_s, \Omega)^2}\right),\text{ as } \text{dist}(\y_s,\Omega)\to\infty.$$
\end{corollary}
To conclude this section, we discuss the robustness of the method against noisy data. Suppose that we only have access to the perturbed far-field operator  $\mathcal{F}_\delta$ satisfying 
  $$\|\mathcal F-\mathcal F_\delta\|\leq   {\delta}\|\mathcal F\|,$$  where $\delta>0$ represents the noise level. Let  $\mathcal{I}_{MOSM,\delta}$ denote the imaging function corresponding to the noisy data $F_\delta$. We obtain the following stability estimate as in Theorem 7 of \cite{Le2022}.
    $$|\mathcal I_{MOSM}(\y_s)-\mathcal{I}_{MOSM,\delta}(\y_s)|\leq(\delta^2+2\delta)\|\mathcal{F}\|^2\|\h\|^2,\quad \text{ for all } \y_s\in \R^3. $$
\section{Numerical study}
\label{se: results}
\subsection{Synthetic  data}
In this section, we present a numerical example to validate the performance of the MOSM with synthetic data. All computations were performed using MATLAB. Synthetic data were generated by solving the direct problem using the spectral solver in \cite{Nguyen2019}, with the following incident wave at wavenumber $k=12$ 
$$\E_{\text{in}}(\x, \di,\q) =\q e^{ik\di \cdot  \x},\quad \q=(\di \times \p)\times \di,\quad \p=(1/\sqrt{3},-1/\sqrt{3},1/\sqrt{3}), \quad  \di \in \mathbb{S}^2.$$
The true scatterer $\Omega$, shown in Figures \ref{fig:simulated}(a) and \ref{fig:simulated}(d), is associated with the following coefficients that satisfy  Assumption I
\begin{align*}
    \xi(\x)=\begin{cases}
A_\xi, & \x\in \Omega\\
0,& \text{otherwise}
    \end{cases},
 \hspace{0.4cm}
\eps_r(\x)=\begin{cases}
A_{\eps_r}, & \x\in \Omega\\
I_3,& \text{otherwise}
    \end{cases},
 \hspace{0.4cm}
{\mu_r^{-1}}(\x)=\begin{cases}
A_{\mu_r^{-1}}, & \x\in \Omega\\
I_3,& \text{otherwise},
    \end{cases},
 \hspace{0.4cm}   \zeta=-\xi,
\end{align*}
with the diagonal matrices  $A_\xi=\text{diag}(0.03, 0.02,  0.01)$, $A_{\eps_r}=\text{diag}(0.8+0.5i,  0.7+i,  0.6+0.4i)$, and $A_{\mu_r^{-1}}=\text{diag}(0.2-0.3i,  0.6-0.4i,  0.9-0.7i)$.

The far-field data $\u^\infty(\widehat{\x},\di,\q)$  were collected at $30$ uniformly distributed points each for $\widehat{\x}$ and $\di$ on $\mathbb{S}^2$. To simulate noise, we incorporate into the data vector a complex-valued vector $\mathcal{N}\in \C^3$ containing numbers $a + ib$,
where $a, b \in (-1, 1)$ randomly generated from a uniform distribution.  The noisy data $\u^\infty_\delta$ is given by
$$\u^\infty_\delta=\u^\infty+\delta\frac{\mathcal{N}}{\|\mathcal{N}\|} \|\u^\infty\|,$$
 where $\delta>0$ is the level of noise. The sampling domain is the cubic $[-1.5,1.5]^3$, with 40 sampling points uniformly distributed on each side. An isovalue of $0.5$ ($50\%$ of the maximal value of the normalized imaging function) was used for 3D isosurface plotting. Color bars are excluded in the visualization, as the focus is on accurately identifying the support of coefficients.

Figure \ref{fig:simulated} shows the reconstruction by the MOSM with $k=12$ under $30\%$ random noise in \ref{fig:simulated}(b), \ref{fig:simulated}(e), and $50\%$ random noise in \ref{fig:simulated}(c), \ref{fig:simulated}(f). The method provides reasonable results for reconstructing the scatterer's shape and location, with slight differences between the results at different noise levels. This demonstrates that the MOSM is stable against noisy data, while also being fast and easy to implement.

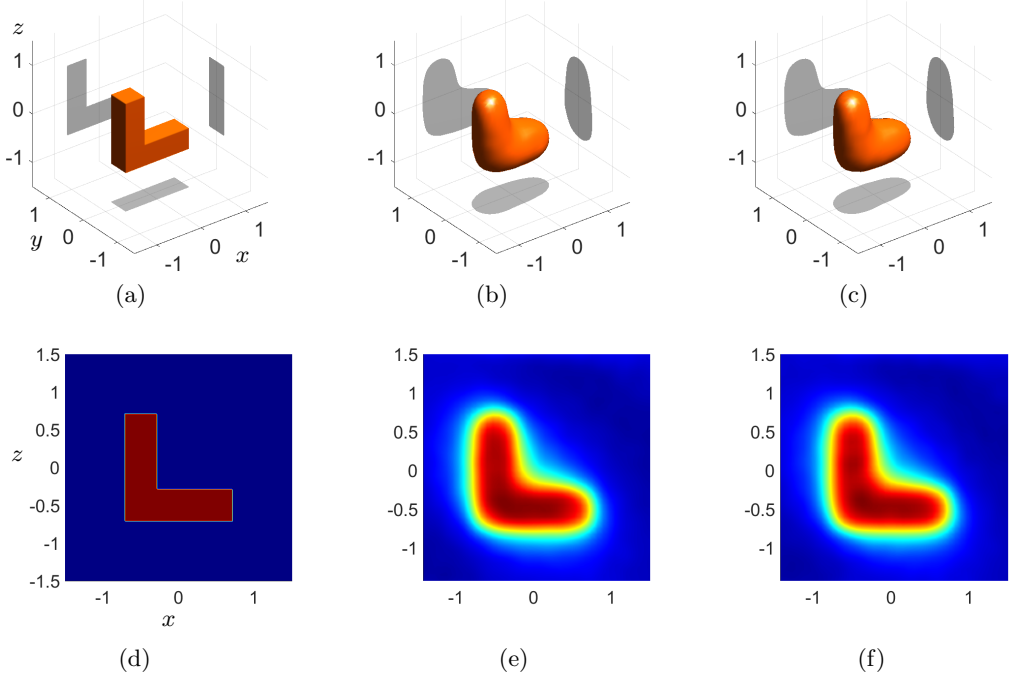
\begin{figure}[h]
    \centering
     \subfloat[]{
    \begin{tikzpicture}
    \node[anchor=south west,inner sep=0] (image) at (0,0) {\includegraphics[width=3.5cm]{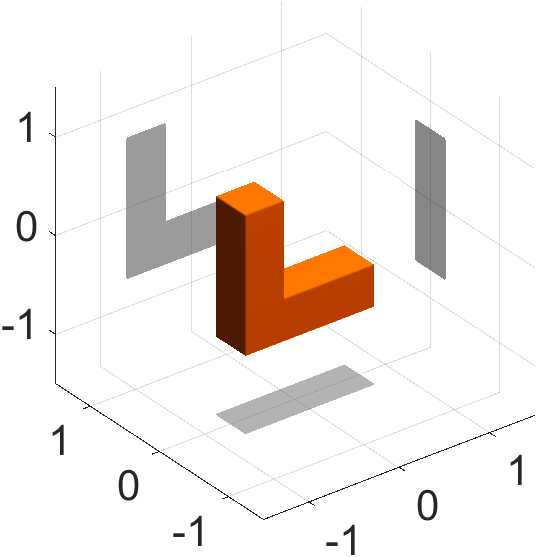}};
    \begin{scope}[x={(image.south east)},y={(image.north west)}]
      \node[anchor=north] at (.9,0.11) {{\footnotesize{$x$}}};
      \node[anchor=south] at (0.12,.04) {\footnotesize{$y$}};
      \node[anchor=south] at (0.05,0.83) {\footnotesize{$z$}};      
    \end{scope}
      \end{tikzpicture}}
           \hspace{1cm}
     \subfloat[]{
    \begin{tikzpicture}
    \node[anchor=south west,inner sep=0] (image) at (0,0) {\includegraphics[width=3.5cm]{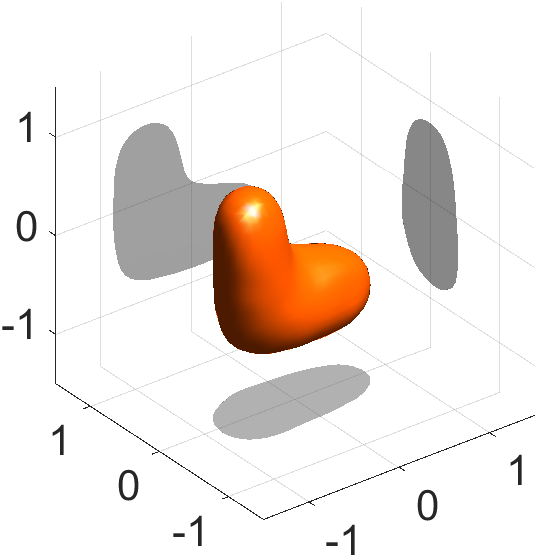}};
    \begin{scope}[x={(image.south east)},y={(image.north west)}]
      \node[anchor=north] at (.9,0.11) {\phantom{\footnotesize{$x$}}};
      \node[anchor=south] at (0.12,.06) {\phantom{\footnotesize{$y$}}};
      \node[anchor=south] at (0.05,0.83) {\phantom{\footnotesize{$x$}}};      
    \end{scope}
  \end{tikzpicture}} 
      \hspace{1cm}
     \subfloat[]{
    \begin{tikzpicture}
    \node[anchor=south west,inner sep=0] (image) at (0,0) {\includegraphics[width=3.5cm]{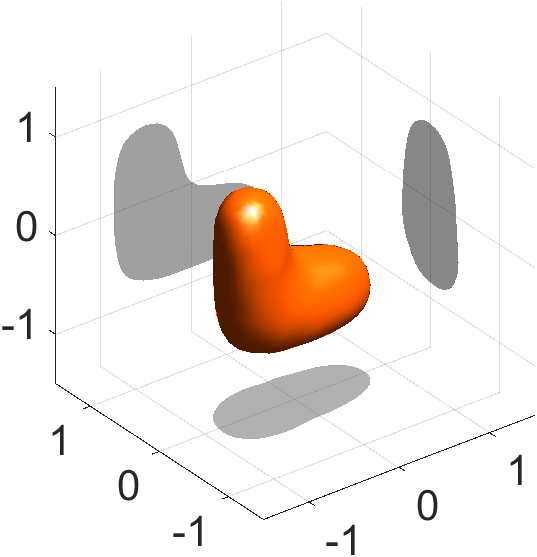}};
    \begin{scope}[x={(image.south east)},y={(image.north west)}]
      \node[anchor=north] at (.9,0.11) {\phantom{\footnotesize{$x$}}};
      \node[anchor=south] at (0.12,.06) {\phantom{\footnotesize{$y$}}};
      \node[anchor=south] at (0.05,0.83) {\phantom{\footnotesize{$x$}}};      
    \end{scope}
  \end{tikzpicture}}  
   \\
  \vspace{0.1cm}
  \subfloat[ ]{
  \begin{tikzpicture}
    \node[anchor=south west,inner sep=0] (image) at (0,0) {\includegraphics[width=3.5cm]{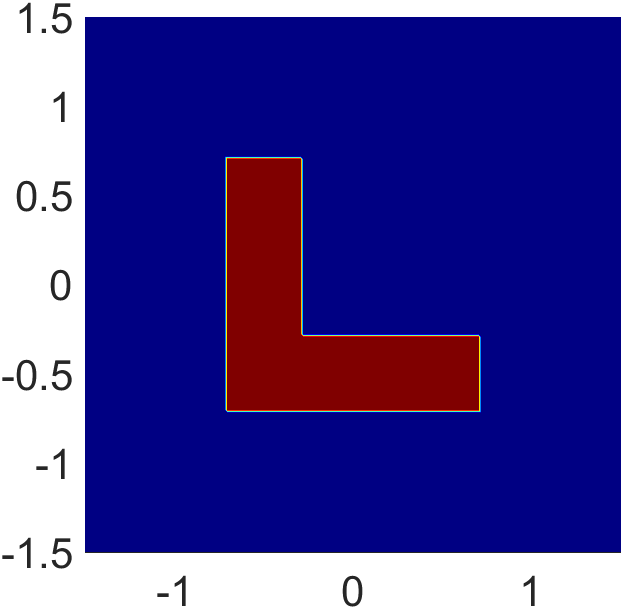}};
    \begin{scope}[x={(image.south east)},y={(image.north west)}]
      \node[anchor=north] at (.53,0) {\footnotesize{$x$}};
      \hspace{0.2cm}
      \node[anchor=south] at (-0.1,.52) {\footnotesize{$z$}};
    \end{scope}
  \end{tikzpicture}}
      \hspace{1cm}
  \subfloat[]{
  \begin{tikzpicture}
    \node[anchor=south west,inner sep=0] (image) at (0,0) {\includegraphics[width=3.5cm]{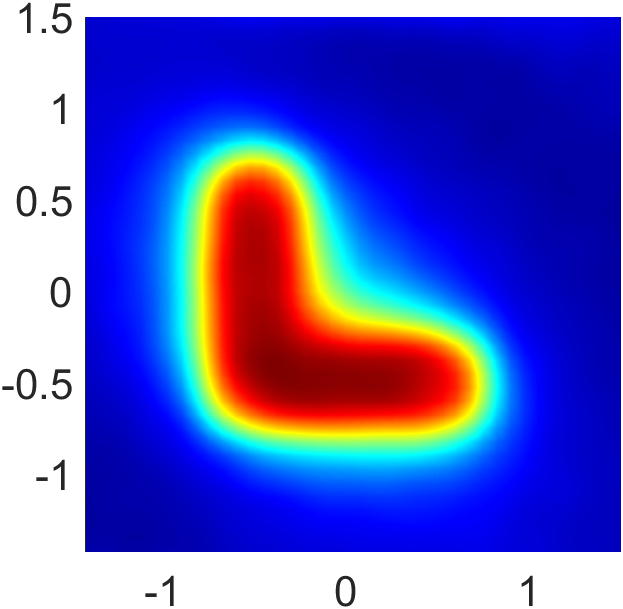}};
    \begin{scope}[x={(image.south east)},y={(image.north west)}]
      \node[anchor=north] at (.5,0) {\phantom{\footnotesize{$x$}}};
    \end{scope}
  \end{tikzpicture}}
  \hspace{1cm}
  \subfloat[]{
  \begin{tikzpicture}
    \node[anchor=south west,inner sep=0] (image) at (0,0) {\includegraphics[width=3.5cm]{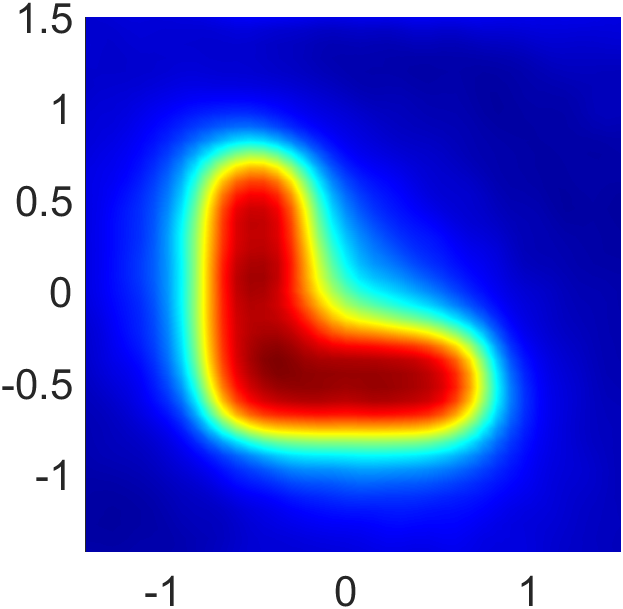}};
    \begin{scope}[x={(image.south east)},y={(image.north west)}]
      \node[anchor=north] at (.5,0) {\phantom{\footnotesize{$x$}}};
    \end{scope}
  \end{tikzpicture}}
  \caption{Reconstruction results of an L-shaped scatterer with synthetic data at $k=12$: (a) true scatterer; (b) reconstruction with $30\%$ noise; (c) reconstruction with $50\%$ noise; (d)--(f) 2D cross-sectional views of (a)--(c) on $\{y=0\}$, respectively.}
       \label{fig:simulated}
  \end{figure}
\subsection{Raw experimental data}
In this section, we validate the performance of the MOSM against unprocessed experimental data provided by the Fresnel Institute, France \cite{Geffrin2009}. This dataset offers a comprehensive range of three-dimensional targets with frequencies spanning $3-8$ GHz and has been widely used to test different inversion methods, see \cite{Litma2009}.  The details of the experimental configurations are available at \cite{Geffrin2009}.  The targets considered here include TwoSpheres, TwoCubes, 
CubeSpheres, and IsoSphere, all with co-polarization. They were illuminated by $81$ incident fields from directions uniformly distributed over a sphere. Due to the technical and practical constraints of the experimental setup, the total fields were measured at only $36$ points uniformly located at the intersection of the sphere and the plane $\{z=0\}$. The noisy scattered field is obtained by subtracting the incident field from the raw total field in the dataset. Then, the far-field pattern can be computed from the scattered field using \eqref{eq:farfield}, since the distance between the targets and the receivers is sufficiently large to ensure that the approximation in \eqref{eq:farfield} holds.
Note that only the third component of the scattered field was given. 
This limited-aperture data presents a significant challenge for traditional sampling methods, which typically require full-aperture data. 

 We rescale $40$ mm to be $1$ unit of length in our simulations. The sampling domain is the cube $[-2.5, 2.5]^3$, with $32$ sampling points uniformly divided on each side. We also replicate the true targets in MATLAB to compare with our reconstruction results. An isovalue of 0.5 was used for 3D isosurface plotting of the real data, consistent with the synthetic data. 

The main challenge when working with this raw experimental dataset is that, since the targets are very small compared to the wavelength, the scattered field is significantly weaker than the measured total and incident fields. This implies that any variations, such as the presence of noise, between two measurements, can completely disrupt the scattered field. At lower frequencies, the data becomes even more sensitive to noise.
To address this, a common approach involves a two-step procedure to clean the data, see \cite{Geffrin2009}. In the first step, drift correction enhances data quality 
by minimizing the spectral bandwidth of the scattered field. However, this approach requires the scattered field to have a limited spectral bandwidth. The second step involves applying calibration, which is less suitable for domain reconstruction methods. In contrast, our MOSM method uses directly raw experimental data without requiring any pre-processing steps, making it more practical and easy to implement.  As shown in Figures \ref{fig:twospheres}--\ref{fig:mystery}, the MOSM method achieves reasonable inversion results with only slight differences compared to results obtained using processed data in \cite{Le2022}, which indicates its high robustness against noisy data. Additionally, its computation is Mast and completed in under 30 seconds.

For each target, we present results using a wavenumber selected to match the one used in \cite{Le2022}, which provides one of the most reasonable reconstructions among those available in the dataset. It is worth noting that our method employs only a single wavenumber.
  {Figure \ref{fig:twospheres} shows MOSM results compared to the classical factorization method (FM) \cite{Kirsch1998} and the original OSM \cite{Harris2020} for the TwoSpheres target.
The MOSM accurately inverts the location and shape of the spheres at $4$ GHz ($k \approx 3.35$). Moreover, consistent with the findings in \cite{Le2022} for processed real 3D data, the MOSM outperforms the other two methods in inverting the unprocessed data.} For the raw CubeSpheres dataset, the proposed method is also able to determine geometrical information about the targets, as shown in Figure \ref{fig:twocubes}.
Next, Figures \ref{fig:cubespheres} and \ref{fig:mystery} illustrate the reconstruction results for the CubeSpheres and IsocaSphere datasets, respectively. Among all the targets, the CubeSpheres dataset features the finest geometric details. Additionally, the IsocaSphere dataset, also referred to as the Mystery object, is the most complex in the database, consisting of twelve small spheres arranged less evenly than those in the CubeSpheres dataset. Despite the sparse, limited-aperture of the real data along the 
$z$-axis and the complexity of the target geometries,
the MOSM can quickly recover the targets’ locations and sizes while providing a rough estimate of their shapes in the $(x,y)$-plane.

The real data corresponds to the case of isotropic media with $\eps_r = \eps_r(x), \mu_r = I_3,\xi=\eta =0$ in $\Omega$, which requires the assumption that the wavenumbers used are not transmission eigenvalues for the analysis of the MOSM. Identifying these eigenvalues is a nontrivial task \cite{Cakoni2016}, and as such, we cannot directly verify this assumption. However, the chance of encountering these transmission eigenvalues is expected to be low, as the set of such eigenvalues is known to be at most discrete.

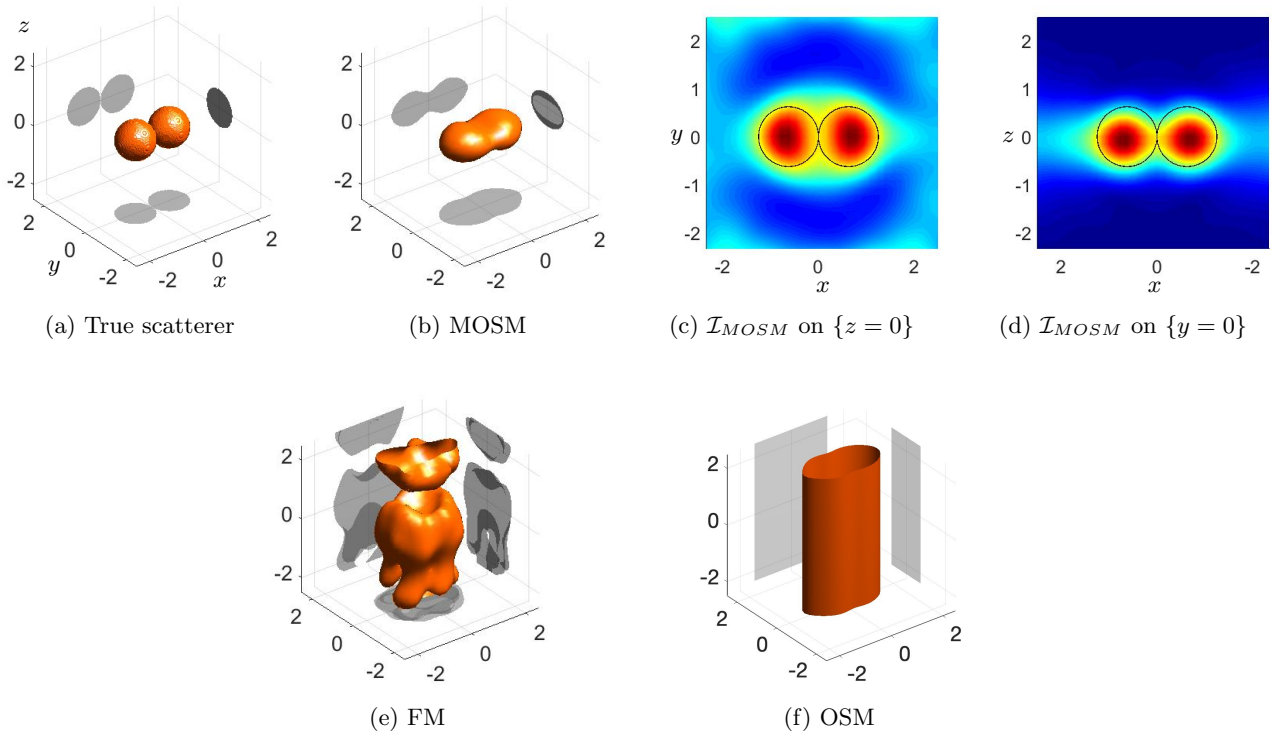
\begin{figure}[h]
        \centering
     \subfloat[True scatterer]{
    \begin{tikzpicture}
    \node[anchor=south west,inner sep=0] (image) at (0,0) {\includegraphics[width=4.cm, trim={3cm 0.cm 3cm 0.cm}, clip]{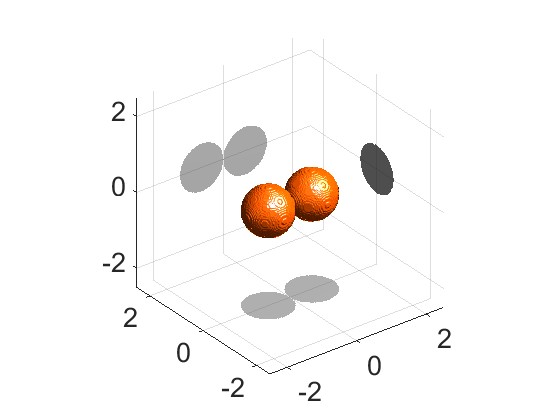}};
    \begin{scope}[x={(image.south east)},y={(image.north west)}]
      \node[anchor=north] at (.75,0.11) {{\footnotesize{$x$}}};
      \node[anchor=south] at (0.2,.05) {\footnotesize{$y$}};
      \node[anchor=south] at (0.1,0.8) {\footnotesize{$z$}};      
    \end{scope}
      \end{tikzpicture}}
          \hspace{0.1cm}
          \subfloat[MOSM]{
    \begin{tikzpicture}
    \node[anchor=south west,inner sep=0] (image) at (0,0) { \includegraphics[width=4.cm, trim={3cm 0.cm 3cm 0.cm}, clip]{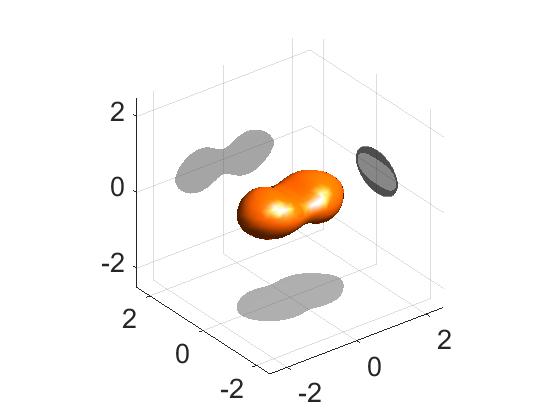}};
    \begin{scope}[x={(image.south east)},y={(image.north west)}]
      \node[anchor=north] at (.75,0.11) {\phantom{\footnotesize{$x$}}};
      \node[anchor=south] at (0.2,.05) {\phantom{\footnotesize{$y$}}};
      \node[anchor=south] at (0.15,0.8) {\phantom{\footnotesize{$z$}}};      
    \end{scope}
      \end{tikzpicture}}
       \hspace{0.05cm}
         \subfloat[$\mathcal{I}_{MOSM}$ on $\{z=0\}$]{
  \begin{tikzpicture}
    \node[anchor=south west,inner sep=0] (image) at (0,0) {\includegraphics[width=3.5cm, trim={3cm 0.cm 3cm 0.cm}, clip]{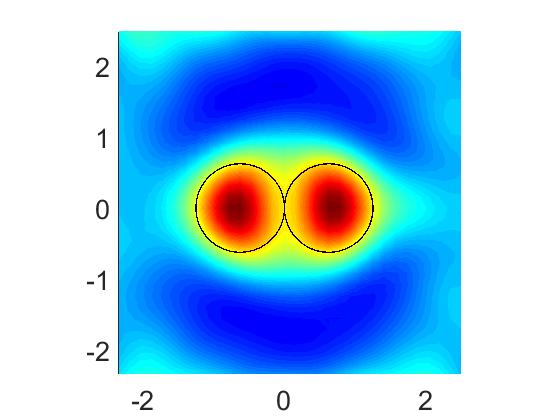}};
    \begin{scope}[x={(image.south east)},y={(image.north west)}]
      \node[anchor=north] at (.53,0.03) {\footnotesize{$x$}};
      \hspace{0.2cm}
      \node[anchor=south] at (-0.08,.44) {\footnotesize{$y$}};
    \end{scope}
  \end{tikzpicture}}
   \hspace{0.15cm}
            \subfloat[$\mathcal{I}_{MOSM}$ on $\{y=0\}$]{
  \begin{tikzpicture}
    \node[anchor=south west,inner sep=0] (image) at (0,0) {\includegraphics[width=3.5cm, trim={3cm 0.cm 3cm 0.cm}, clip]{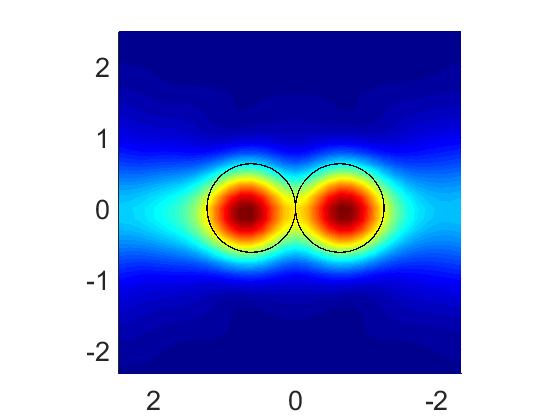}  };
    \begin{scope}[x={(image.south east)},y={(image.north west)}]
      \node[anchor=north] at (.53,0.03) {\footnotesize{$x$}};
      \hspace{0.2cm}
      \node[anchor=south] at (-0.08,.45) {\footnotesize{$z$}};
    \end{scope}
  \end{tikzpicture}}     
    \\
  \subfloat[FM]{
    \begin{tikzpicture}
    \node[anchor=south west,inner sep=0] (image) at (0,0) { \includegraphics[width=4.cm, trim={3cm 0.cm 3cm 0.cm}, clip]{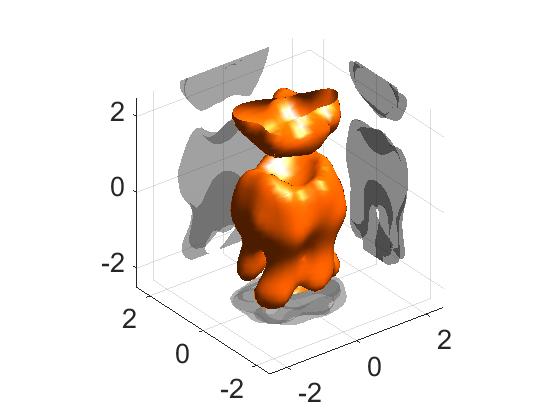}};
    \begin{scope}[x={(image.south east)},y={(image.north west)}]
      \node[anchor=north] at (.75,0.11) {\phantom{\footnotesize{$x$}}};
      \node[anchor=south] at (0.2,.05) {\phantom{\footnotesize{$y$}}};
      \node[anchor=south] at (0.15,0.8) {\phantom{\footnotesize{$z$}}};      
    \end{scope}
      \end{tikzpicture}}
      \hspace{1cm}
 \subfloat[OSM]{
    \begin{tikzpicture}
    \node[anchor=south west,inner sep=0] (image) at (0,0) { \includegraphics[width=4.7cm, trim={3cm 0.cm 3cm 0.cm}, clip]{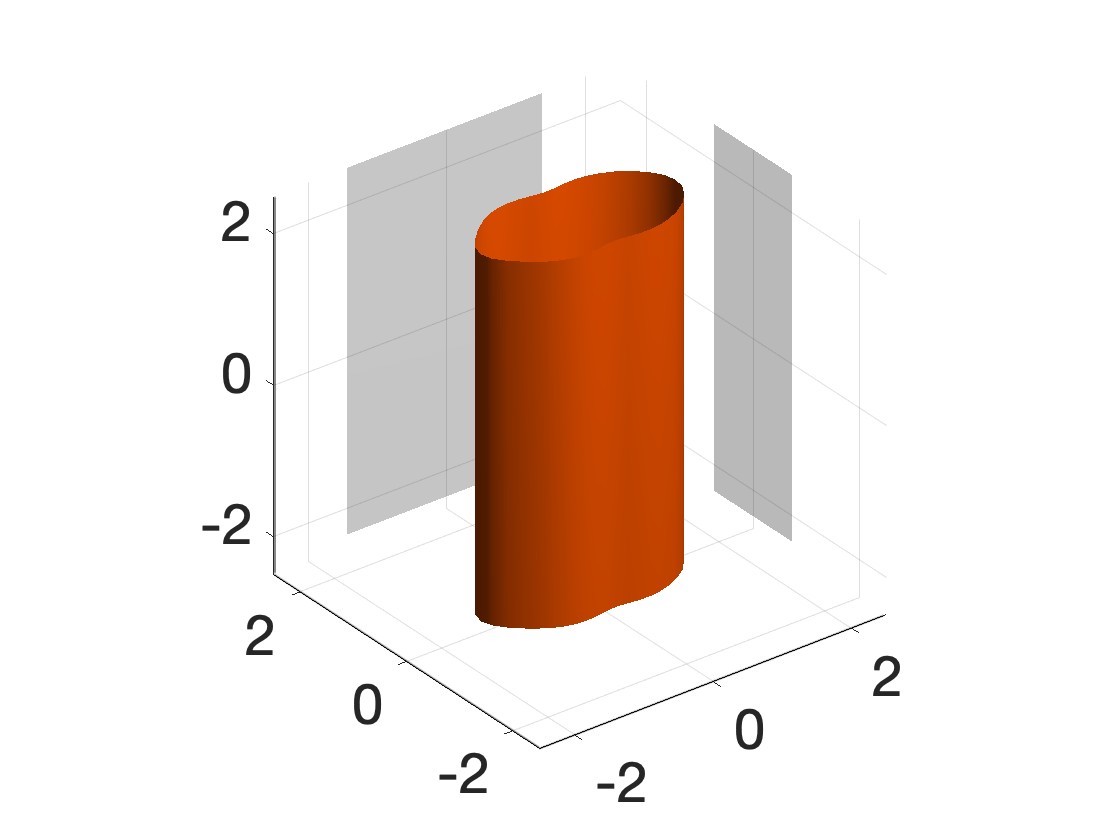}};
    \begin{scope}[x={(image.south east)},y={(image.north west)}]
      \node[anchor=north] at (.75,0.11) {\phantom{\footnotesize{$x$}}};
      \node[anchor=south] at (0.2,.05) {\phantom{\footnotesize{$y$}}};
      \node[anchor=south] at (0.15,0.8) {\phantom{\footnotesize{$z$}}};      
    \end{scope}
      \end{tikzpicture}}
       \caption{  {MOSM, FM and OSM reconstruction of TwoSpheres using raw experimental data at $4$ GHz. }}
       \label{fig:twospheres}
   \end{figure}   
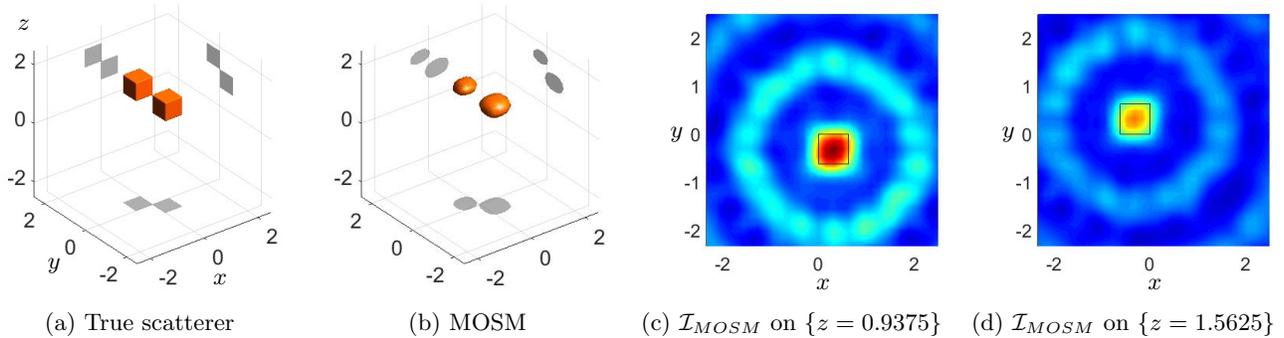
\begin{figure}[h]
        \centering
     \subfloat[True scatterer]{
    \begin{tikzpicture}
    \node[anchor=south west,inner sep=0] (image) at (0,0) {\includegraphics[width=4.cm, trim={3cm 0.cm 3cm 0.cm}, clip]{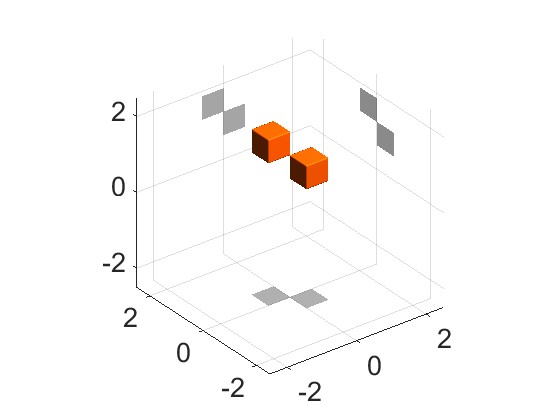}};
    \begin{scope}[x={(image.south east)},y={(image.north west)}]
      \node[anchor=north] at (.75,0.11) {{\footnotesize{$x$}}};
      \node[anchor=south] at (0.2,.05) {\footnotesize{$y$}};
      \node[anchor=south] at (0.1,0.8) {\footnotesize{$z$}};      
    \end{scope}
      \end{tikzpicture}}
          \hspace{0.1cm}
          \subfloat[MOSM]{
    \begin{tikzpicture}
    \node[anchor=south west,inner sep=0] (image) at (0,0) { \includegraphics[width=4.cm, trim={3cm 0.cm 3cm 0.cm}, clip]{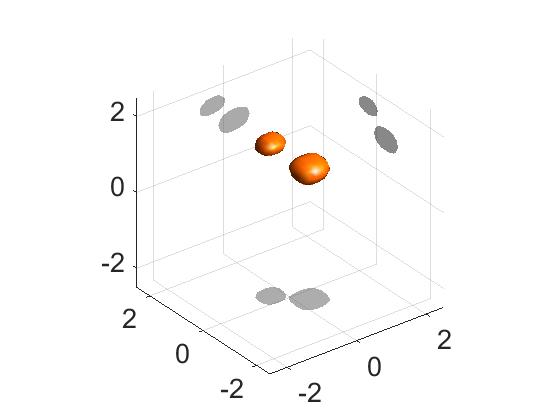}};
    \begin{scope}[x={(image.south east)},y={(image.north west)}]
      \node[anchor=north] at (.75,0.11) {\phantom{\footnotesize{$x$}}};
      \node[anchor=south] at (0.2,.05) {\phantom{\footnotesize{$y$}}};
      \node[anchor=south] at (0.15,0.8) {\phantom{\footnotesize{$z$}}};      
    \end{scope}
      \end{tikzpicture}}
       \hspace{0.05cm}
         \subfloat[$\mathcal{I}_{MOSM}$ on $\{z=0.9375\}$]{
  \begin{tikzpicture}
    \node[anchor=south west,inner sep=0] (image) at (0,0) {\includegraphics[width=3.5cm, trim={3cm 0.cm 3cm 0.cm}, clip]{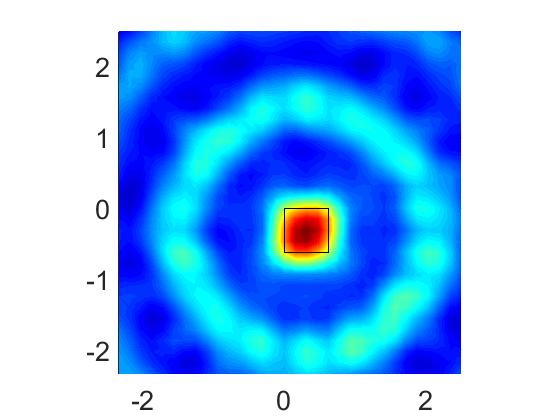} };
    \begin{scope}[x={(image.south east)},y={(image.north west)}]
      \node[anchor=north] at (.53,0.03) {\footnotesize{$x$}};
      \hspace{0.2cm}
      \node[anchor=south] at (-0.08,.44) {\footnotesize{$y$}};
    \end{scope}
  \end{tikzpicture}}
           \hspace{0.15 cm}
            \subfloat[$\mathcal{I}_{MOSM}$ on $\{z=1.5625\}$]{
  \begin{tikzpicture}
    \node[anchor=south west,inner sep=0] (image) at (0,0) {\includegraphics[width=3.5cm, trim={3cm 0.cm 3cm 0.cm}, clip]{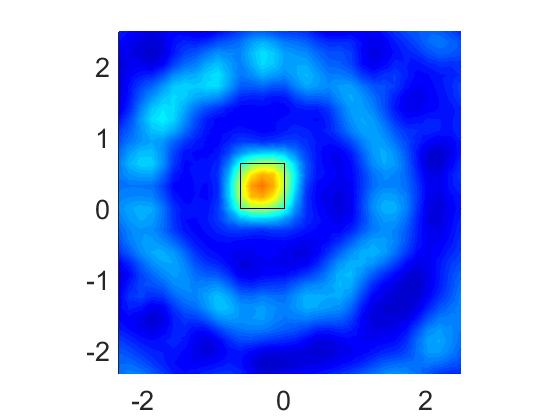}  };
    \begin{scope}[x={(image.south east)},y={(image.north west)}]
      \node[anchor=north] at (.53,0.03) {\footnotesize{$x$}};
      \hspace{0.2cm}
      \node[anchor=south] at (-0.08,.44) {\footnotesize{$y$}};
    \end{scope}
  \end{tikzpicture}}
       \caption{MOSM reconstruction of TwoCubes using raw experimental data at $7.5$ GHz. }
      \label{fig:twocubes}
   \end{figure}   
\begin{figure}[h]
        \centering
     \subfloat[True scatterer]{
    \begin{tikzpicture}
    \node[anchor=south west,inner sep=0] (image) at (0,0) {\includegraphics[width=4.cm, trim={3cm 0.cm 3cm 0.cm}, clip]{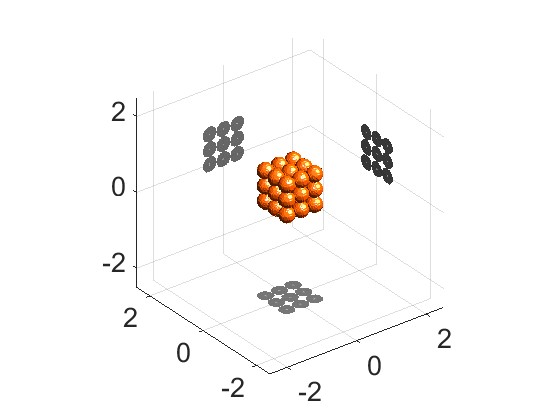}};
    \begin{scope}[x={(image.south east)},y={(image.north west)}]
      \node[anchor=north] at (.75,0.11) {{\footnotesize{$x$}}};
      \node[anchor=south] at (0.2,.05) {\footnotesize{$y$}};
      \node[anchor=south] at (0.1,0.8) {\footnotesize{$z$}};      
    \end{scope}
      \end{tikzpicture}}
          \hspace{0.1cm}
          \subfloat[MOSM]{
    \begin{tikzpicture}
    \node[anchor=south west,inner sep=0] (image) at (0,0) { \includegraphics[width=4.cm, trim={3cm 0.cm 3cm 0.cm}, clip]{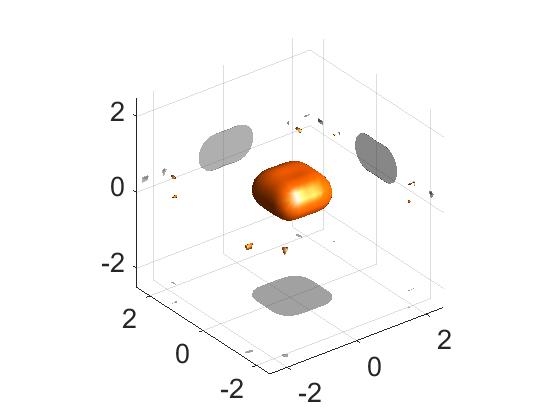}};
    \begin{scope}[x={(image.south east)},y={(image.north west)}]
      \node[anchor=north] at (.75,0.11) {\phantom{\footnotesize{$x$}}};
      \node[anchor=south] at (0.2,.05) {\phantom{\footnotesize{$y$}}};
      \node[anchor=south] at (0.15,0.8) {\phantom{\footnotesize{$z$}}};      
    \end{scope}
      \end{tikzpicture}}
       \hspace{0.05cm}
            \subfloat[$\mathcal{I}_{MOSM}$ on $\{z=0.3975\}$]{
  \begin{tikzpicture}
    \node[anchor=south west,inner sep=0] (image) at (0,0) {\includegraphics[width=3.5cm, trim={3cm 0.cm 3cm 0.cm}, clip]{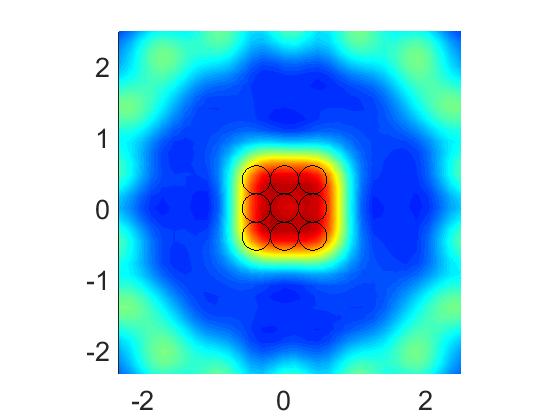}  };
    \begin{scope}[x={(image.south east)},y={(image.north west)}]
      \node[anchor=north] at (.53,0.03) {\footnotesize{$x$}};
      \hspace{0.2cm}
      \node[anchor=south] at (-0.08,.44) {\footnotesize{$y$}};
    \end{scope}
      \end{tikzpicture}}
       \hspace{0.15cm}
            \subfloat[$\mathcal{I}_{MOSM}$ on $\{y=0\}$]{
  \begin{tikzpicture}
    \node[anchor=south west,inner sep=0] (image) at (0,0) {\includegraphics[width=3.5cm, trim={3cm 0.cm 3cm 0.cm}, clip]{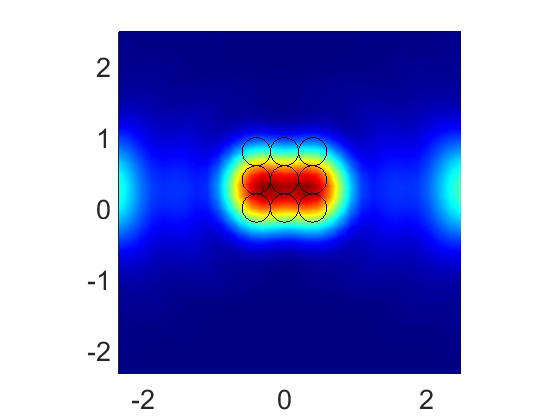}  };
    \begin{scope}[x={(image.south east)},y={(image.north west)}]
      \node[anchor=north] at (.53,0.03) {\footnotesize{$x$}};
      \hspace{0.2cm}
      \node[anchor=south] at (-0.08,.44) {\footnotesize{$z$}};
    \end{scope}
      \end{tikzpicture}}
       \caption{MOSM reconstruction of CubeSpheres using raw experimental data at $4.75$ GHz. }
       \label{fig:cubespheres}
   \end{figure}
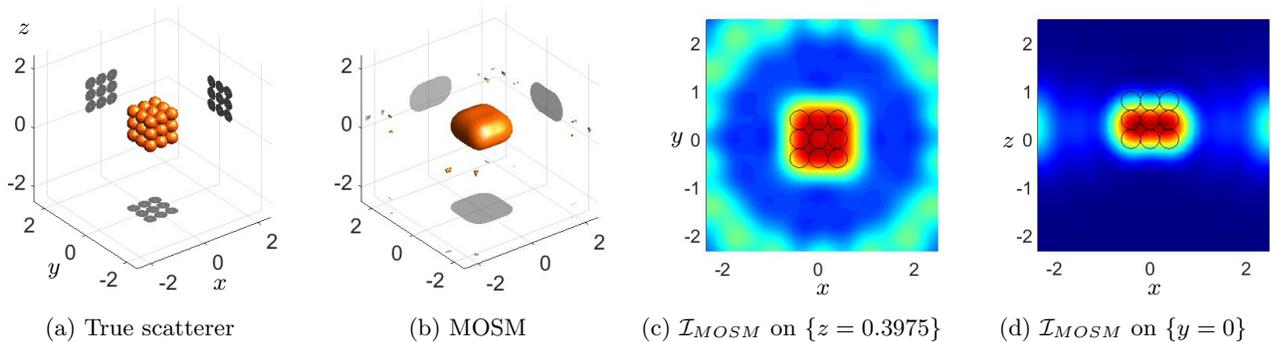  
\begin{figure}[h]
        \centering
     \subfloat[True scatterer]{
    \begin{tikzpicture}
    \node[anchor=south west,inner sep=0] (image) at (0,0) {\includegraphics[width=4.cm, trim={3cm 0.cm 3cm 0.cm}, clip]{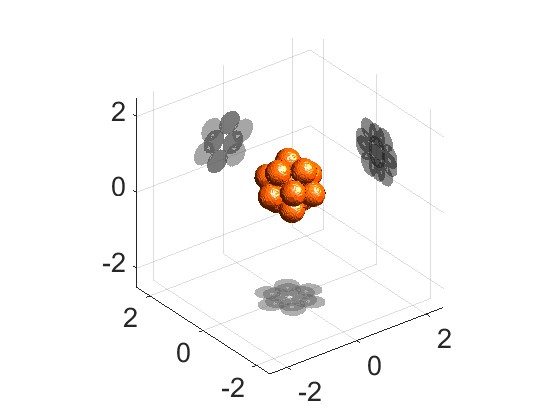}};
    \begin{scope}[x={(image.south east)},y={(image.north west)}]
      \node[anchor=north] at (.75,0.11) {{\footnotesize{$x$}}};
      \node[anchor=south] at (0.2,.05) {\footnotesize{$y$}};
      \node[anchor=south] at (0.1,0.8) {\footnotesize{$z$}};      
    \end{scope}
      \end{tikzpicture}}
          \hspace{0.1cm}
          \subfloat[MOSM]{
    \begin{tikzpicture}
    \node[anchor=south west,inner sep=0] (image) at (0,0) { \includegraphics[width=4.cm, trim={3cm 0.cm 3cm 0.cm}, clip]{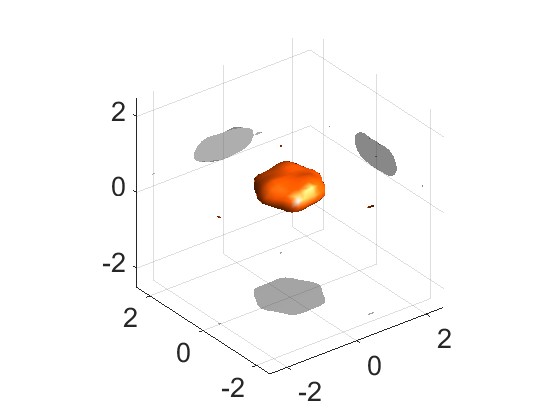}};
    \begin{scope}[x={(image.south east)},y={(image.north west)}]
      \node[anchor=north] at (.75,0.11) {\phantom{\footnotesize{$x$}}};
      \node[anchor=south] at (0.2,.05) {\phantom{\footnotesize{$y$}}};
      \node[anchor=south] at (0.15,0.8) {\phantom{\footnotesize{$z$}}};      
    \end{scope}
      \end{tikzpicture}}
       \hspace{0.05cm}
            \subfloat[$\mathcal{I}_{MOSM}$ on $\{z=0.3435\}$]{
  \begin{tikzpicture}
    \node[anchor=south west,inner sep=0] (image) at (0,0) {\includegraphics[width=3.5cm, trim={3cm 0.cm 3cm 0.cm}, clip]{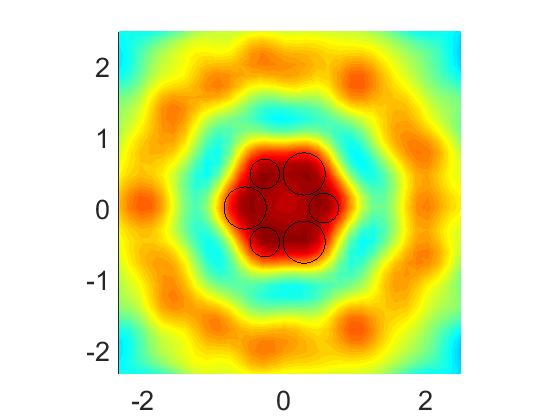}  };
    \begin{scope}[x={(image.south east)},y={(image.north west)}]
      \node[anchor=north] at (.53,0.03) {\footnotesize{$x$}};
      \hspace{0.2cm}
      \node[anchor=south] at (-0.08,.44) {\footnotesize{$y$}};
    \end{scope}
      \end{tikzpicture}}
       \hspace{0.15cm}
            \subfloat[$\mathcal{I}_{MOSM}$ on $\{z=0\}$]{
  \begin{tikzpicture}
    \node[anchor=south west,inner sep=0] (image) at (0,0) {\includegraphics[width=3.5cm, trim={3cm 0.cm 3cm 0.cm}, clip]{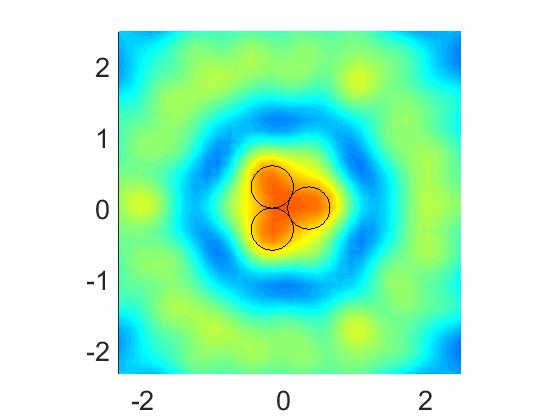}  };
    \begin{scope}[x={(image.south east)},y={(image.north west)}]
      \node[anchor=north] at (.53,0.03) {\footnotesize{$x$}};
      \hspace{0.2cm}
      \node[anchor=south] at (-0.08,.44) {\footnotesize{$y$}};
    \end{scope}
      \end{tikzpicture}}
       \caption{MOSM reconstruction of IsocaSphere using raw experimental data at $5.75$ GHz.}
       \label{fig:mystery}
   \end{figure}
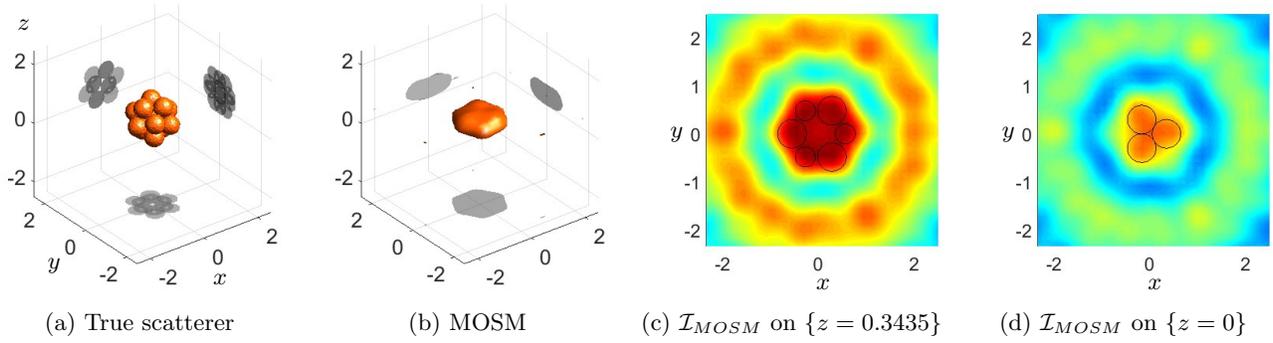 
   \section{Conclusion}
   \label{se: conclude}
{ Through the factorization analysis of the far-field operator, we establish the uniqueness of the inverse problem for Maxwell's equations in reconstructing bianisotropic targets from multi-static far-field data. The factorization analysis also helps justify the resolution and stability of the MOSM. The sampling method's resolution is within the diffraction limit. Numerical experiments on synthetic data with varying noise levels confirm the method’s effectiveness and robustness. Moreover, we validate the sampling method on unprocessed 3D experimental datasets from the Fresnel Institute. Although the proposed method theoretically requires  full aperture data, it can be directly applied to these sparse, limited-aperture datasets without requiring any pre-processing. These results indicate that the proposed method is practical, easy and fast to implement, while maintaining robustness and accuracy even in the presence of high noise levels in the data. }

\vspace{0.5cm}
\textbf{Acknowledgment.} This work was partially supported by NSF Grant DMS-2208293.

{\footnotesize
\bibliographystyle{plain}
\bibliography{Imylib}
}
\end{document}